\documentclass[11pt]{amsart}
\usepackage[foot]{amsaddr}
\usepackage[utf8]{inputenc}
\usepackage{amsmath}
\usepackage{amsfonts}
\usepackage{amssymb}
\usepackage{amsthm}
\usepackage{geometry}
\usepackage{enumitem}
\usepackage[usenames,dvipsnames,svgnames]{xcolor}
\usepackage{hyperref}
\usepackage[english]{babel}
\usepackage[symbol]{footmisc}
\usepackage{fancyhdr}
\usepackage[centercolon]{mathtools}
\usepackage{microtype}
\usepackage{lmodern}
\usepackage[only,llbracket,rrbracket]{stmaryrd}
\usepackage{tikz}
\usepackage{dsfont}

\def\Xint#1{\mathchoice
{\XXint\displaystyle\textstyle{#1}}%
{\XXint\textstyle\scriptstyle{#1}}%
{\XXint\scriptstyle\scriptscriptstyle{#1}}%
{\XXint\scriptscriptstyle\scriptscriptstyle{#1}}%
\!\int}
\def\XXint#1#2#3{{\setbox0=\hbox{$#1{#2#3}{\int}$ }
\vcenter{\hbox{$#2#3$ }}\kern-.59\wd0}}
\def\avint{\Xint-}

\newcommand{\edge}{{\scriptstyle\mid}}
\newcommand{\grad}{\nabla}

\newcommand{\cF}{{\mathcal F}}
\newcommand{\cM}{{\mathcal M}}

\newcommand{\R}{\mathbb{R}}

\newcommand{\Prob}{\mathbb{P}}
\newcommand{\EX}{\mathbb{E}}

\newcommand{\dsOne}{\mathds{1}}

\DeclareMathOperator{\diam}{diam}

\newcommand{\Ha}{\ensuremath{\mathcal{H}}}
\newcommand{\D}{\ensuremath{\mathcal{D}}}

\newcommand{\cP}{\mathcal{P}}
\newcommand{\io}{\int_{\Omega}}

\newcommand{\llb}{\llbracket}
\newcommand{\rrb}{\rrbracket}

\newtheorem{theorem}{Theorem}
\newtheorem{lemma}{Lemma}
\newtheorem{remark}{Remark}
\newtheorem{definition}{Definition}

\DeclarePairedDelimiter{\abs}{\lvert}{\rvert}
\DeclarePairedDelimiter{\norm}{\lVert}{\rVert}
\DeclarePairedDelimiter{\bra}{(}{)}
\DeclarePairedDelimiter{\pra}{[}{]}
\DeclarePairedDelimiter{\set}{\{}{\}}

\newcommand{\cT}{\mathcal{T}}

\newcommand{\BigO}{\mathcal{O}}

\DeclareMathOperator{\I}{I}
\DeclareMathOperator{\II}{II}
\DeclareMathOperator{\III}{III}
\DeclareMathOperator{\IV}{IV}

\definecolor{darkblue}{rgb}{0,0,0.6}

\title[Error estimates for a FV-scheme for advection-diffusion equations]{Error estimates for a finite volume scheme for advection-diffusion equations with rough coefficients}

\author{Víctor Navarro-Fernández \and André Schlichting}

\email{\{victor.navarro,a.schlichting\}@uni-muenster.de}
\address{Institut f\"ur Analysis und Numerik,  Westf\"alische Wilhelms-Universit\"at M\"unster \newline
Orl\'eans-Ring~10, 48149 M\"unster, Germany.}

\allowdisplaybreaks[1]

\begin{document}

\date{\today}

\begin{abstract}
We study the implicit upwind finite volume scheme for numerically approximating the advection-diffusion equation with a vector field in the low regularity DiPerna-Lions setting. That is, we are concerned with advecting velocity fields that are spatially Sobolev regular and data that are merely integrable. We prove that on unstructured regular meshes the rate of convergence of approximate solutions generated by the upwind scheme towards the unique solution of the continuous model is at least one. The numerical error is estimated in terms of logarithmic Kantorovich–Rubinstein distances and provides a bound on the rate of weak convergence.
\end{abstract}

\maketitle

\section{Introduction}\label{S:intro}

The advection-diffusion equation is of great relevance in a wide range of different scientific fields. Also referred to as the Fokker--Planck equation or convection-diffusion equation, it appears related to the Navier--Stokes equation in fluid dynamics, to the Black--Scholes equation in financial mathematics, in semiconductor physics, in biology or engineering. It describes the transport of a scalar quantity $\theta\in\R$ under the effect of a vector field $u\in\R^d$ and in the presence of diffusion~\cite{BKRS2015}.

In this paper we are concerned with a bounded domain $\Omega\subset\R^d$, a bounded time interval $(0,T)$ and a positive constant diffusion coefficient $\kappa>0$. Given vector field $u:[0,T]\times \Omega \to \R^d$, we study the evolution of a scalar quantity $\theta: [0,T)\times \Omega \to \R$ described by the Cauchy problem
\begin{equation}\label{eq:advdiff}
	\left\lbrace
	\begin{array}{rcll}
		\partial_t\theta + \nabla\cdot(u\theta) & = & \kappa\Delta\theta & \text{ in } (0,T)\times\Omega, \\
		\theta(0,\cdot) & = & \theta^0 & \text{ in } \Omega,
	\end{array}
	\right.
\end{equation}
where $\theta^0$ is the initial configuration. 

In addition we assume that there is no loss of mass across the boundary of the domain,
\begin{equation}\label{eq:noflow}
	\bra*{\kappa \nabla \theta - u} \cdot n = 0 \quad  \text{in } (0,T) \times \partial \Omega ,
\end{equation}
where $n=n(x)$ represents the outer unit vector normal to the boundary of the domain on every point $x\in\partial\Omega$. This assumption implies that solutions to the advection-diffusion equation \eqref{eq:advdiff} conserve their mass in time,
\[
\io \theta(t,x)dx = \io \theta^0(x)dx \quad \text{for all } t\in (0,T).
\]

Well-posedness of solutions for smooth vector fields and initial data goes back to the classical theory of parabolic equations, see Lady\v zenskaja et al. \cite{Lady68}. In some specific contexts in physics, for instance, when studying the transport of a mass, dye, or any scalar quantity by a turbulent flow~\cite{Ottino1990,ShraimanSiggia2000}, the vector field involved has a very low regularity, thus a mathematical theory for transport and advection-diffusion equations with rough vector fields is needed.

In this context, well-posedness of renormalized solutions to the equation \eqref{eq:advdiff} is obtained for Sobolev regular vector fields by DiPerna and Lions \cite{DiPernaLions89}. This new solution concept is based on the least possible regularity such that the chain rule still holds, providing qualitative stability and hence uniqueness results. We say a vector field $u$ is in the DiPerna-Lions setting if for some $1<p\leq \infty$ it holds
\begin{equation}\label{eq:diperna-lions}
	u\in L^1((0,T);W^{1,p}(\Omega)) \qquad \text{and} \qquad (\nabla\cdot u)^-\in L^1((0,T);L^\infty(\Omega)),
\end{equation}
where we write $f^-=\max\{0,-f\}$ referring to the negative part of a function $f$. For works explicitly handling diffusion in this regularity setting, see also~\cite{BKRS2015,LeBrisLions19,Figalli08,LeBrisLions08}.

Considering then $\theta^0\in L^q$ with $q>1$ such that
$1/p+1/q \leq 1$, 
there is a unique distributional solution to the advection-diffusion equation \eqref{eq:advdiff} with vector field in the DiPerna-Lions setting such that
\[
\theta \in L^\infty((0,T);L^q(\Omega))\cap L^1((0,T);W^{1,1}(\Omega)).
\]
Such regularity for the solution to \eqref{eq:advdiff} can be straightforwardly derived from the standard apriori estimate
\begin{equation}\label{eq:apriori}
\frac{1}{q(q-1)}\frac{d}{dt}\|\theta\|_{L^q}^q + \kappa\int_\Omega|\theta|^{q-2}|\nabla\theta|^2 dx \leq \frac{1}{q}\|(\nabla\cdot u)^-\|_{L^\infty} \|\theta\|_{L^q}^q,
\end{equation}
with $q>1$. Then one can see that the solution begin $L^\infty((0,T);L^q(\Omega))$ is obtained by integrating \eqref{eq:apriori} and dropping the term with $\nabla\theta$ so that we get
\begin{equation}\label{eq:cont_stability1}
	\|\theta\|_{L^\infty(L^q)} \leq \Lambda^{1-\frac{1}{q}} \|\theta^0\|_{L^q},
\end{equation}
where $\Lambda = \exp(\|(\nabla\cdot u)^-\|_{L^1(L^\infty)})$ is the compressibility constant of the vector field. Here and in the following we use the Bochner space notation $L^r(L^s)$ to denote the space $L^r((0,T);L^s(\Omega))$ and similarly for other Banach spaces.
Since we dropped the term with $\nabla\theta$ in order to get \eqref{eq:cont_stability1}, the estimate holds for both the transport equation ($\kappa=0$) and the advection-diffusion equation ($\kappa>0$). However, the presence of diffusion provides better regularity for the solution, which is obtained from the term involving $\nabla\theta$ in \eqref{eq:apriori} as
\begin{equation}\label{eq:cont_stability2}
\kappa\int_0^T \int_\Omega |\theta|^{q-2}|\nabla\theta|^2 dxds \leq \frac{1}{q}\left( \frac{1}{q-1}+\Lambda^{q-1}\log\Lambda \right)\|\theta^0\|_{L^q}^q.
\end{equation}
This control over the gradient provides that the solution to \eqref{eq:advdiff} lives in $L^1((0,T);W^{1,1}(\Omega))$ (see the beginning of Section~\ref{S:proof}). %

Throughout this paper, we will work with vector fields in the DiPerna-Lions setting. However, solutions can also be defined for other types of non-smooth vector fields. 
For instance, there exists the so-called Lady\v zenskaja--Prodi--Serrin setting, recently revisited in~\cite{BianchiniColomboCrippaSpinolo17}, where solutions are well-posed for vector fields with no control on the spatial derivatives but instead additional time integrability. 
Ambrosio~\cite{Ambrosio04} proved well-posedness for vector fields with bounded variation regularity, $u\in L^1(BV)$. 
Some other rough settings could be considered, for example, when $u$ is a singular integral of an $L^1$ function, see \cite{CrippaNobiliSeisSpirito17,NavarroFernandezSchlichtingSeis21}, which is of special interest for fluid dynamics. 
Nonetheless, it is remarkable that out of the DiPerna-Lions setting, well-posedness is not guaranteed, even with diffusion. If $p$ and $q$ are such that $1/p+1/q > 1+1/d$, Modena, Sattig and Székelihidi proved nonuniqueness of solutions in \cite{ModenaSattig20,ModenaSzekelyhidi18}. 
Whether solutions are well-posed or not in the gap between DiPerna-Lions and Modena-Sattig-Székelihidi remains an open problem in the class $L^\infty(L^q)$. 
However, Cheskidov and Luo \cite{CheskidovLuo21} very recently proved nonuniqueness for solutions in the class $L^1(L^q)$ for every exponent such that $1/p+1/q>1$ at the expense of a worse time-integrability.

The main objective of this paper is to develop (optimal) error estimates for an upwind scheme on unstructured meshes based on a finite volume approximation of distributional solutions to the advection-diffusion equation \eqref{eq:advdiff} when the vector field is in the DiPerna-Lions setting. This result arises as a continuation of the works by Schlichting and Seis \cite{SchlichtingSeis17, SchlichtingSeis18}, where the authors study the upwind scheme for the transport equation, i.e., $\kappa = 0$, in a similar regularity setting. The addition of a diffusive term is not trivial whatsoever, as we will explain in detail along the sections of this paper. 
The main result in \cite[Theorem 1]{NavarroFernandezSchlichtingSeis21} provides the stability estimate in the presence of diffusion, which will be a key ingredient for derivation of error estimates for the numerical scheme.
In the DiPerna-Lions setting, the stability for two solutions is measured with respect to the optimal transport distance defined for any $\delta >0$ by
\begin{equation}\label{def:intro:Dlog}
	\D_\delta(\mu_1,\mu_2) = \inf_{\pi\in\Pi(\mu_1,\mu_2)}\iint_{\Omega\times\Omega} \log\left(\frac{|x-y|}{\delta}+1\right)d\pi(x,y).
\end{equation}
Here $\Pi(\mu_1,\mu_2)$ represents the set of all transport plans between the measures $\mu_1$ and $\mu_2$. We give a more in-depth contextualization and further explanation about these so-called \emph{Kantorovich-Rubinstein distances} in Section \ref{S:optimal_transport}.

The result \cite[Theorem 1]{NavarroFernandezSchlichtingSeis21} states that any two solutions $\theta_1$ and $\theta_2$ of the advection-diffusion equation \eqref{eq:advdiff} with initial data $\theta^0_1$, $\theta^0_2$, vector fields $u_1$, $u_2$ and diffusion coefficients $\kappa_1$, $\kappa_2$ respectively, satisfy
\begin{equation}\label{eq:cont_stab_est}
	\sup_{0\leq t\leq T}\mathcal{D}_\delta(\theta_1(t),\theta_2(t)) \lesssim \mathcal{D}_\delta(\theta^0_1,\theta^0_2) + 1 + \frac{\|u_1-u_2\|_{L^1(L^p)}+|\kappa_1-\kappa_2| \|\nabla\theta_2\|_{L^1}}{\delta}.
\end{equation}
Hereby, we are using the notation $a\lesssim b$ to express that there exists a constant $C>0$ only depending on the norms in assumption~\eqref{eq:diperna-lions}, the initial data and the domain $\Omega$ such that $a\leq Cb$.

The study of convergence rates for finite volume schemes for the advection-diffusion equation is intimately related to the study of the diffusionless case, that was firstly addressed by Kuznetsov \cite{Kuznetsov76}. For results about the mathematical theory and the derivation of optimal error estimates with Lipschitz vector fields and regular initial data, either $BV$ or $H^1$, see \cite{DelarueLagoutiere11,Merlet07,VilaVilledieu03}. On the DiPerna-Lions setting, the problem has not been addressed until very recently with the work of Schlichting and Seis, first with Cartesian meshes \cite{SchlichtingSeis17} and after with unstructured meshes \cite{SchlichtingSeis18}.

The study of numerical approximations for the advection equation with diffusion has been of great interest along the last decades, from classical results with $C^1(\overline{\Omega})$ vector fields and Cartesian meshes \cite{Samarskii65,SamarskiiLazarovMakarov87,TikhonovSamarskii62} to less regular settings \cite{BertolazziManzini04,CaiMandelMcCormick91,Heinrich87,OllivierGoochVanAltena02,Vanselow96}.

The novelty in our work is that we present error estimates for the finite volume scheme for the advection-diffusion in the low regularity framework. Denoting by $h$ to the size of the mesh and $k$ to the time step, we get an $\BigO(h+\sqrt{k})$ error bound as in the smooth setting. We derive most of the results and estimates here working in the Eulerian setting for the equation \eqref{eq:advdiff}, that is, operating with the solution of the partial differential equation. In previous works, mainly for the transport equation, the Lagrangian setting has been considered instead, i.e., the characteristics associated with the equation. This provides a probabilistic interpretation of the numerical scheme as a Markov chain on the mesh (see \cite{DelarueLagoutiereVauchelet16,SchlichtingSeis17}). In this paper, we need to use the Lagrangian setting to prove one estimate related to the time-discretization of the vector field. Since we are dealing with a parabolic equation, the characteristics are solutions to stochastic differential equation, for which reason we include a short introduction to Lagrangian stochastic flows in Appendix~\ref{s:SLF}.

In addition, it is remarkable that working in a low regularity setting carries over a substantial change in topology compared to the smooth setting. Here we quantify the rate of \emph{weak convergence}, following the spirit of previous works for the transport equation, e.g.\ \cite{DelarueLagoutiereVauchelet16,SchlichtingSeis17,SchlichtingSeis18}. For Lipschitz vector fields instead, it is possible to derive bounds in strong norms. However, for the DiPerna-Lions setting, we introduce the Kantorovich--Rubinstein distance that metrize weak convergence and hence it is a natural tool for studying this case, since only for those stability estimates are available~\cite{NavarroFernandezSchlichtingSeis21}.

In this work, we focus on the upwind finite volume scheme for linear advection, since it is the easiest to analyze and has the needed stability properties. An interesting question is, if the here presented proofs generalize to the analysis of structure preserving schemes for singular aggregation-diffusion equations, like the ones studied for regular aggregation in~\cite{DelarueLagoutiereVauchelet2017,SchlichtingSeis2022}.

\emph{This paper is organized as follows:} In Section \ref{S:setting} we present a precise definition of the admissible meshes, the finite volume numerical scheme, and its properties together with a presentation and a discussion of the main results. In Section \ref{S:optimal_transport} we introduce the logarithmic Kantorovich–Rubinstein distance that plays a pivotal role in the results here presented. Section \ref{S:proof} contains all the proofs related to the main result of this paper. Finally, Appendix \ref{s:SLF} provides an overview of stochastic Lagrangian flows on bounded domains, which is a needed tool to estimate the error related to the time-discretization of the vector field.

\section{Setting and main result}\label{S:setting}

\subsection{Definition of the numerical scheme}\label{ss:scheme_definition}

In this Section we present a formal and detailed definition of the upwind scheme that we will use. To begin with, recall from \cite{EymardGallouetHerbin00} the definition of admissible meshes for the finite volume discretization of advection-diffusion equations.

\begin{definition}[Admissible meshes]\label{def:admissible_mesh}
Let $\Omega\subset\R^d$ be an open, locally convex and bounded set with $C^{1,1}$ boundary. We say $\cT$ is an admissible tessellation of $\Omega$ if it consists of a finite family of cells or control volumes $K\in\cT$ and a finite family of points $\{x_K\}_{K\in\cT}\subset\overline\Omega$ such that
\begin{itemize}
\item every control volume $K\in \cT$ is a closed, connected and convex subset in $\Omega$;
\item the control volumes have disjoint interiors and satisfy
$\overline{\Omega} = \bigcup_{K\in\cT} K;$
\item each cell is polygonal in the interior of $\Omega$, in the sense that the interior boundary of each cell $\partial K\setminus\partial\Omega$ is the union of finitely many subsets of $\Omega$ contained in hyperplanes of $\R^d$ with strictly positive $\mathcal{H}^{d-1}$-measure;
\item the family of points $\{x_K\}_{K\in\cT}$ satisfies $x_K\in \overline{K}\setminus\partial\Omega$ for all $K\in\cT$;.
\end{itemize}
\end{definition}

In general, see \cite{EymardGallouetHerbin00}, the geometry of $\partial\Omega$ is restricted to the case in which it is polygonal itself. 
However in our specific case, we need a construction of a \emph{stochastic Lagrangian flow} (see Appendix~\ref{s:SLF}), for which certain error terms can only be controlled on domains satisfying a \emph{uniform exterior ball condition}~\eqref{ass:exteriorball}, for which a $C^{1,1}$ boundary is a sufficient condition. 
Since, we are working under a no-flow boundary condition \eqref{eq:noflow}, we can indeed consider sufficiently smooth domains $\Omega$ such that Definition \ref{def:admissible_mesh} holds and the numerical cells are only polygonal inside of the domain $\Omega$.

It is important to remark that the convexity requirement for the cells is needed in our analysis in order to prove Lemma \ref{lemma:iv} invoking a specific construction, the Brenier maps. Nonetheless we believe that this might not be strictly needed in general and one could come up with an similar construction that allows some relaxation for the convexity assumption.

A two dimensional example of two admissible control volumes is illustrated in Figure \ref{fig:mesh}. 
We denote by $L\sim K$ whenever $K$ and $L$ are two neighbouring cells and we write  $K\edge L$ to denote the common edge. If $L\sim K$, we define $d_{KL} = |x_L-x_K|$ and $n_{KL}$ to be the unit vector on $K\edge L$ pointing in the direction $x_L-x_K$. In addition, abusing the notation, we write $|K\edge L| = \mathcal{H}^{d-1}(K\edge L)$ the $(d-1)$-dimensional Hausdorff measure of the edge $K\cap L$ and $|K|=\mathcal{L}^d(K)$ the $d$-dimensional Lebesgue measure of a cell $K\in\cT$.

\begin{figure}\centering
\tikzset{every picture/.style={line width=0.75pt}} %

\begin{tikzpicture}[x=0.66pt,y=0.66pt,yscale=-0.9,xscale=0.9]
\draw   (138,143) -- (207.79,46.94) -- (320.71,83.63) -- (320.71,202.37) -- (207.79,239.06) -- cycle ;
\draw    (320.71,83.63) -- (367,34) ;
\draw    (357,226) -- (320.71,202.37) ;
\draw   (231.15,134.97) -- (241.41,145.1)(241.26,134.99) -- (231.29,145.08) ;
\draw   (402.15,134.97) -- (412.41,145.1)(412.26,134.99) -- (402.29,145.08) ;
\draw    (471,191) -- (357,226) ;
\draw    (471,191) -- (471,80) ;
\draw    (367,34) -- (471,80) ;
\draw  [dash pattern={on 4.5pt off 4.5pt}]  (245,140) -- (401,140) ;
\draw   (321,129) -- (332,129) -- (332,140) ;

\draw (270,120) node [anchor=north west][inner sep=0.75pt]   [align=left] {$d_{KL}$};
\draw (324,157) node [anchor=north west][inner sep=0.75pt]   [align=left] {$K\edge L$};
\draw (226,237) node [anchor=north west][inner sep=0.75pt]   [align=left] {$K$};
\draw (409,216) node [anchor=north west][inner sep=0.75pt]   [align=left] {$L\sim K$};
\draw (205,142) node [anchor=north west][inner sep=0.75pt]   [align=left] {$x_K$};
\draw (417,142) node [anchor=north west][inner sep=0.75pt]   [align=left] {$x_L$};

\end{tikzpicture}
\caption{Example of admissible neighbouring control volumes}
\label{fig:mesh}
\end{figure}
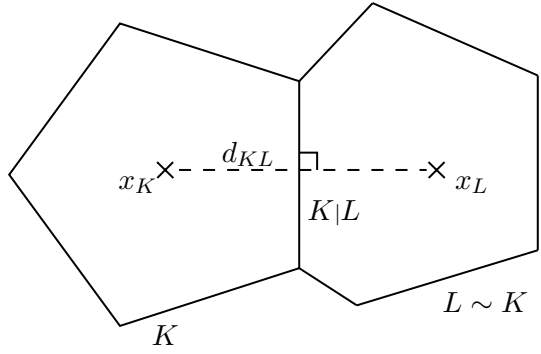

The \emph{mesh size} $h$ is defined to be the maximal cell diameter,
$
h = \max_{K\in\cT}\diam K,
$ 
and hence it holds $d_{KL}\lesssim h$ for all $K\in\cT$ and $L\sim K$. For the time discretization we call $k$ the \emph{time step} such that there exists $N\in\mathbb{N}$ with $T=kN$ and we adopt the convention $t_j=jk$ for all $0\leq j\leq N$. For the sake of shorter notation, we write $\llb 0,N \rrb$ to denote the collection of numbers $\{0,1,\hdots,N\}$.

In addition it is required to consider some regularity assumptions for the boundary of the domain and the mesh to ensure that, at least, the standard geometric constants arising on the Poincaré and the trace inequalities do not depend on the size of the mesh. Namely, it is needed that for every $f\in W^{1,1}(K)\cap C(\overline K)$,
\begin{equation}\label{eq:regularity_mesh}
\begin{aligned}
\|f\|_{L^1(\partial K)} & \lesssim \|\grad f\|_{L^1(K)} + h^{-1} \|f\|_{L^1(K)}, \\
\norm{f - f_K}_{L^1(K)} & \lesssim  h \norm{\nabla f}_{L^1(K)},
\end{aligned}
\end{equation}
uniformly in $K\in \cT$ and $h>0$. These are respectively the trace and Poincaré inequalities and for a classical proof of these results we refer to \cite[Sections 4.3 and 4.5]{EvansGariepy2015}. We denote by $f_K$ the average of $f$ over the cell $K$, to be more specific
$
f_K = \avint_K fdx.
$
One direct consequence of the trace estimate is the so-called \emph{isoperimetric property} of the mesh, that guarantees that every cell $K$ of the tessellation has a volume of order $h^d$ and a surface of order $h^{d-1}$, and reads as follows
\begin{equation}\label{eq:isoperimetric}
\frac{|\partial K|}{|K|} \lesssim \frac{1}{h}.
\end{equation}

In Definition \ref{def:admissible_mesh} we assumed the boundary of $\Omega$ to be $C^{1,1}$, i.e. $C^1$ with Lipschitz derivative. This requirement is sufficient because with such regularity $\partial\Omega$ satisfies the \emph{uniform exterior ball condition}: For some $r_0>0$ and for all $x\in \partial \Omega$ it holds
\begin{equation}\label{ass:exteriorball}
	\forall y \in \overline \Omega\setminus\set{x}: \frac{x-y}{\abs{x-y}}\cdot n(x) + \frac{1}{2r_0} \abs*{x-y} \geq 0 . 
\end{equation}

In order to define explicitly the numerical scheme that we are considering here, we first need to approximate the initial datum. Since the finite volume scheme approximates the solution by averaging on every cell, we can consider the discretization of the initial datum in this way,
\begin{equation}\label{eq:datum_discrete}
\theta_K^0 = \avint_K\theta^0\, dx
\end{equation}
and hence $\theta_h^0(x) = \theta_K^0$ for every $x\in K$ and every $K\in\cT$. 
Since the scheme considers net fluxes across the cell faces, we define the discretized normal velocity from a control volume $K$ to a neighboring one $L\sim K$ by
\begin{equation}\label{eq:u_discrete}
u_{KL}^n = \avint_{t^n}^{t^{n+1}} \avint_{K\edge L} u\cdot n_{KL} d\Ha^{d-1} dt.
\end{equation}
Both $u_{KL}^n$ and $\theta_K^0$ are well-defined thanks to the trace theorem for Sobolev vector fields, i.e.~\eqref{eq:regularity_mesh}. Notice that by definition the discretization of the velocity is antisymmetric with respect to the control volumes, i.e. it holds $u_{KL}^n=-u_{LK}^n$, which is useful for many calculations.

Making use again of the notation $f^+$ and $f^-$ to denote the positive and negative part of a function, we define the finite volume scheme for the advection-diffusion equation \eqref{eq:advdiff} as
\begin{equation}\label{eq:scheme1}
\frac{\theta_K^{n+1} - \theta_K^n}{k} + \sum_{L\sim K} \frac{\abs{K\edge L}}{\abs{K}} (u_{KL}^{n+}\theta_K^{n+1} - u_{KL}^{n-}\theta_L^{n+1}) + \kappa \sum_{L\sim K} \frac{\abs{K\edge L}}{\abs{K}} \frac{\theta_K^{n+1}-\theta_L^{n+1}}{d_{KL}} = 0
\end{equation}
for every $n\in \llb 0,N-1 \rrb$ and $K\in\cT$. Therefore the \emph{approximate solution} $\theta_{k,h}$ is defined by
\[
\theta_{k,h}(t,x) = \theta_K^n \quad \text{for almost every } (t,x)\in [t^n,t^{n+1})\times K
\]
for every $n\in \llb 0,N-1 \rrb$ and $K\in\cT$. If $n=0$ we directly define $\theta_{k,h}^0=\theta_h^0$. 

Within the next section we show that this numerical problem is well-posed (see Lemma \ref{lemma:wellposedness}) and we will derive analogous stability estimate to \eqref{eq:cont_stability1} and \eqref{eq:cont_stability2} (see Lemma \ref{lemma:stability_estimate}). These results follow under the assumption that the time step verifies $k\leq k_{\max}$.
The definition of the \emph{maximal time step} $k_{\max}$ follows a similar construction as in~\cite{Boyer12,SchlichtingSeis18}, where it is given depending on some $\alpha>1$ as the smaller number $k_{\max}=k_{\max}(\alpha)$ such that
\begin{equation}\label{eq:max_time_step}
\frac{q-1}{q}\int_I \|(\nabla\cdot u)^-\|_{L^\infty}dt \leq \frac{\alpha-1}{\alpha}
\qquad\forall I\subseteq [0,T) \text{ with } |I|\leq k_{\max}(\alpha).
\end{equation}
The constant $\alpha>1$ is used as a measure of how close the numerical solution $\theta_{k,h}$ is from satisfying the a priori estimate~\eqref{eq:cont_stability1}. 
Indeed, we will see in Lemma \ref{lemma:stability_estimate} that the exponent $1-1/q$ on the compressibility constant is replaced by $\alpha(1-1/q)$ and thus $\alpha=1$ for incompressible vector fields, i.e. if $\nabla\cdot u = 0$.

\subsection{Main result}\label{ss:main_result}

The main result here presented concerns an estimate for the error generated by the finite volume scheme \eqref{eq:scheme1} as an approximation of the advection-diffusion equation \eqref{eq:advdiff}. Without further ado let us recall the precise hypotheses we need for Theorem \ref{theorem}. First of all we consider a vector field in the DiPerna-Lions setting for some $p\in(1,\infty]$. Then assume the initial datum is integrable with
\begin{equation}\label{eq:initial_datum}
\theta^0\in L^q(\Omega) \qquad\text{with } q\in (1,\infty] \quad\text{ such that }\quad
\frac{1}{p}+\frac{1}{q} \leq 1.
\end{equation}
Last, for the numerical analysis we need to consider bounded vector fields,
\begin{equation}\label{eq:vector_field_bounded}
u\in L^\infty((0,T)\times\Omega).
\end{equation}
Although this is not required for the derivation of the continuous stability estimates \eqref{eq:cont_stability1}--\eqref{eq:cont_stability2}, it is a standard and not very restrictive assumption for numerical experiments, see for instance~\cite{SchlichtingSeis18}. The main result therefore states as follows.

\begin{theorem}\label{theorem}
Consider $\theta^0$, $u$ and $k_{\max}$ such that \eqref{eq:diperna-lions}, \eqref{eq:max_time_step},\eqref{eq:initial_datum} and \eqref{eq:vector_field_bounded}  hold. Consider an admissible tessellation of $\Omega$ that satisfies \eqref{eq:regularity_mesh}. Let $\theta$ be the unique distributional solution to \eqref{eq:advdiff}--\eqref{eq:noflow} and for $k\in (0,\min\{k_{\max},1\})$ and $h\in(0,1)$ let $\theta_{k,h}$ be the unique approximate solution given by the numerical scheme \eqref{eq:scheme1}. Then, for any $\delta >0$ it holds
\begin{equation}\label{eq:principal_estimate}
\sup_{0\leq t\leq T} \mathcal{D}_\delta(\theta(t),\theta_{k,h}(t)) \lesssim 1 +\frac{h}{\delta}\min\set[\Bigg]{ \norm{u}_\infty \sqrt{\frac{T}{\kappa}} , \sqrt{\frac{T\norm{u}_\infty}{h}}} + \frac{\sqrt{k}}{\delta} \bra*{\sqrt{T}\norm{u}_\infty+\sqrt{\kappa}} . 
\end{equation}
\end{theorem}
Here, $\D_\delta(\cdot,\cdot)$ defined as in \eqref{def:intro:Dlog}, refers to a distance from the theory of optimal transportation with a logarithmic cost. This particular distance is of great use for equations with rough coefficients because it metrizes weak convergence. Namely, if we choose $\delta=h+\sqrt{k}$ then we get
\[
\theta_{k,h} \rightharpoonup \theta \qquad \text{as } k,h\rightarrow 0
\]
weakly with rate (at most) $\delta=h+\sqrt{k}$, we elaborate more on the properties of the Kantorovich--Rubinstein distance in Section \ref{S:optimal_transport}.

The estimate \eqref{eq:principal_estimate} here is presented in such form to make explicit that it is a generalization of the result for transport equations ($\kappa=0$) in \cite{SchlichtingSeis18}. One can appreciate that assuming $\kappa>0$, Theorem \ref{theorem} reads
\[
\sup_{0\leq t\leq T} \mathcal{D}_\delta(\theta(t),\theta_{k,h}(t)) \lesssim 1+\frac{h\norm{u}_\infty}{\delta} \sqrt{\frac{T}{\kappa}} + \frac{\sqrt{k}}{\delta} \bra*{\sqrt{T}\norm{u}_\infty+\sqrt{\kappa}},
\]
which yields an error of order $\BigO(h+\sqrt{k})$ as stated before.

Another consequence of Theorem \ref{theorem} is that in the limit $\kappa\to 0$ the estimate~\eqref{eq:principal_estimate} takes the form
\[
\sup_{0\leq t\leq T} \mathcal{D}_\delta(\theta(t),\theta_{k,h}(t)) \lesssim 1 +\frac{\sqrt{hT\norm{u}_\infty}}{\delta} + \frac{\sqrt{kT}\norm{u}_\infty}{\delta} ,
\]
and thus recovers the result from \cite{SchlichtingSeis18}, i.e.\ an error of order $\BigO(\sqrt{h}+\sqrt{k})$.

For the diffusionless transport  equation, the main source of the spatial discretization error is the phenomenon of numerical diffusion. That is, the numerical scheme \emph{acts like} creating a diffusion term with diffusion coefficient $h$, that for the transport equation reads as an advection-diffusion equation of the form
\[
\partial_t\theta + \nabla\cdot(u\theta) = h\Delta\theta.
\]
Therefore one would expect the rate of convergence for $h$ to be of order $1/2$ since that is the known optimal rate for the vanishing diffusion or inviscid limit case, see \cite{Seis18}. However, with $\kappa>0$, numerical diffusion acts modifying the already existing diffusion coefficient $\kappa$ to $\kappa+h$
\[
\partial_t\theta + \nabla\cdot(u\theta) = (\kappa+h)\Delta\theta,
\]
and hence by the recent result \cite{NavarroFernandezSchlichtingSeis21}, the expected rate of convergence for $h$ has to be of order $1$.

Let us close the discussion of the main result, by remarking on including source-sink distributions and non-homogeneous boundary conditions for~\eqref{eq:advdiff}. First, a flux boundary condition where~\eqref{eq:noflow} is replaced by $\bra*{\kappa \nabla \theta - u} \cdot n =g$ in $(0,T)\times \partial \Omega$ for some $g: [0,T]\times \partial \Omega \to \R$ can be transformed into a source-sink distribution $f:[0,T]\times \Omega \to \R$ in the domain using a suitable extension. 
Hence, it is sufficient to do the analysis instead of~\eqref{eq:advdiff} for
\begin{equation*}\label{eq:advdiff:general}
	\partial_t\theta + \nabla\cdot(u\theta) = \kappa\Delta\theta + f \qquad \text{ in } (0,T)\times\Omega .
\end{equation*}
with suitable initial data and no-flux boundary condition~\eqref{eq:noflow}. By another standard transformation, which consists of renormalizing the density $\theta$, we can ensure that the total sources and sinks are balanced, which amounts to $\int_\Omega f(t,x) dx = 0$ for all $t\in [0,T]$. In particular, these transformations ensure that $\theta$ conserves mass. This is essential for using the optimal transport distance $\mathcal{D}_\delta$ to compare the solution $\theta$ with its numerical approximation $\theta_{k,h}$. 
		
This situation with a balanced source-sink distribution was investigated in~\cite{SchlichtingSeis18} for the diffusionless transport equation, and we expect that the analysis carries over to the present case with diffusion. 
The source-sink term will introduce additional discretization errors when we discuss the discretization of data in Lemmas~\ref{lemma:disc_time}--\ref{lemma:disc_vectorfield} below.
For the temporal discretization, where we use a stochastic Lagrangian representation of the solution, becomes more involved in the presence of a source-sink distribution and we omit it for the sake of concise presentation.

\subsection{Properties of the numerical scheme}\label{ss:scheme_properties}

First of all we state a result on the well-posedness of the numerical scheme. 
\begin{lemma}\label{lemma:wellposedness}
Under the hypothesis for Theorem \ref{theorem}, there exists a unique solution to the implicit upwind scheme \eqref{eq:scheme2} that is mass preserving and monotone, i.e. the solution remains positive for positive initial data.
\end{lemma}
This is a classical result and we refer to \cite[Theorem 4.1]{EymardGallouetHerbin00} for a detailed proof.

The main aim of this section is to develop stability estimates for the numerical scheme that are analogous to the a priori estimates \eqref{eq:cont_stability1} and \eqref{eq:cont_stability2}. In order to do so it is convenient first to recall that some of the discretized versions of the functions involved on the scheme are controlled by their continuous counterpart. Specifically, recall that for the initial datum $\theta^0_h$ and the divergence of the velocity field $\nabla\cdot u$ it holds
\begin{align}\label{eq:bound_datum}
\|\theta^0_h\|_{L^q} &\leq \|\theta^0\|_{L^q},\\
\label{eq:bound_field}
\|(\nabla\cdot u)_{k,h}^-\|_{L^1(L^\infty)} &\leq \|(\nabla\cdot u)^-\|_{L^1(L^\infty)}.
\end{align}
We omit the proof for the sake of brevity but it can be found on \cite[Lemma 3]{SchlichtingSeis18}.

Let us now rewrite the upwind scheme~\eqref{eq:scheme1} in the following equivalent form:
\begin{equation}\label{eq:scheme2}
	\begin{aligned}
		\frac{\theta^{n+1}_K - \theta^{n}_K}{k} + \sum_{L\sim K} \frac{\abs{K\edge L}}{\abs{K}} u_{KL}^n \frac{\theta_K^{n+1} + \theta_L^{n+1}}{2} + \sum_{L\sim K} & \frac{\abs{K\edge L}}{\abs{K}} \abs*{u_{KL}^n}  \frac{\theta_K^{n+1} - \theta_L^{n+1}}{2} \\
		& + \kappa \sum_{L\sim K} \frac{\abs{K\edge L}}{\abs{K}} \frac{\theta_K^{n+1}-\theta_L^{n+1}}{d_{KL}} = 0
	\end{aligned}
\end{equation}
for every $n\in \llb 0,N-1 \rrb$ and $K\in\cT$. This is a straightforward consequence of the identities
\[
u_{KL}^{n+} = \frac{|u_{KL}|+u_{KL}}{2} \qquad \text{and} \qquad u_{KL}^{n-} = \frac{|u_{KL}|-u_{KL}}{2}.
\]
Then, the stability estimates for the finite volume scheme hold as follows.
\begin{lemma}[Stability estimates]\label{lemma:stability_estimate}
Let $\theta_{k,h}$ be the solution to the upwind scheme \eqref{eq:scheme1} with nonnegative initial data. Then for any $q\in (1,\infty)$, $\alpha>1$ and $k\leq k_{\max}(\alpha)$ as defined in \eqref{eq:max_time_step}, it holds
\begin{equation}\label{eq:stability_estimate1}
\|\theta_{k,h}\|_{L^\infty(L^q)} \leq \Lambda_{k,h}^{\alpha\left(1-\frac{1}{q}\right)} \|\theta_h^0\|_{L^q}
\end{equation}
where $\Lambda_{k,h} = \exp(\|(\nabla\cdot u)_{k,h}^-\|_{L^1(L^\infty)})$. Moreover, if $r\in (1,\min\{q,2\}]$ it also holds,
\begin{equation}\label{eq:stability_estimate2}
\begin{aligned}
\sum_n & \sum_K |K|\left(\frac{\theta_K^{n+1}+\theta_K^n}{2}\right)^{r-2}(\theta_K^{n+1}-\theta_K^n)^2 \\
 & + k\sum_n\sum_K\sum_{L\sim K}|K\edge L|\left(|u_{KL}^n|+\frac{\kappa}{d_{KL}}\right)(\theta_K^{n+1}-\theta_L^{n+1})^2\left(\frac{\theta_K^{n+1}+\theta_L^{n+1}}{2}\right)^{r-2} \\
 & \leq C_r\left(1+(r-1)\log\Lambda_{k,h}\right)\Lambda_{k,h}^{\alpha(r-1)} \|\theta_h^0\|_{L^r}^r
\end{aligned}
\end{equation}
with $C_r$ being a positive constant that satisfies $C_r\rightarrow \infty$ as $r\rightarrow 1$.
\end{lemma}
\begin{proof}
By the monotonicity of the scheme and the nonnegativity of the initial datum, we deduce that the solution of the numerical scheme $\theta_{k,h}$ is nonnegative. In order to study those stability estimates we will work with the second formulation of the upwind scheme \eqref{eq:scheme2}. First of all, let us multiply the scheme by $|K|$ so that we get
\[
\begin{aligned}
|K|(\theta^{n+1}_K - \theta^{n}_K) & + k\sum_{L\sim K} \abs{K\edge L} u_{KL}^n \frac{\theta_K^{n+1} + \theta_L^{n+1}}{2} \\
& + k\sum_{L\sim K} \abs{K\edge L} \abs*{u_{KL}^n}  \frac{\theta_K^{n+1} - \theta_L^{n+1}}{2} + \kappa k \sum_{L\sim K} \abs{K\edge L} \frac{\theta_K^{n+1}-\theta_L^{n+1}}{d_{KL}} = 0.
\end{aligned}
\]
We denote the four addends as $\I_K^n + \II_K^n + \III_K^n + \IV_K^n = 0$. Analogously to the continuous setting, we will obtain the stability estimates by testing with $(\theta_K^{n+1})^{q-1}$ and summing over $K\in \cT$, namely
\[
\underbrace{\sum_K \I_K^n(\theta_K^{n+1})^{q-1}}_{\I^n} + \underbrace{\sum_K \II_K^n(\theta_K^{n+1})^{q-1}}_{\II^n} + \underbrace{\sum_K \III_K^n(\theta_K^{n+1})^{q-1}}_{\III^n} + \underbrace{\sum_K \IV_K^n(\theta_K^{n+1})^{q-1}}_{\IV^n} = 0.
\]

For the first term, we can apply Hölder's inequality,
\[
\I^n = \sum_K |K|(\theta_K^{n+1})^q - \sum_K |K|\theta_K^n (\theta_K^{n+1})^{q-1} \geq \|\theta_{k,h}^{n+1}\|_{L^q}^q - \|\theta_{k,h}^n\|_{L^q}\|\theta_{k,h}^{n+1}\|_{L^q}^{q-1}.
\]
For the second term we recall that $u_{KL}^n = -u_{LK}^n$, hence we can symmetrize $\II^n$ as follows
\[
\II^n = \frac{k}{2}\sum_K\sum_{L\sim K}|K\edge L|u_{KL}^n\frac{\theta_K^{n+1}+\theta_L^{n+1}}{2}((\theta_K^{n+1})^{q-1}-(\theta_L^{n+1})^{q-1}).
\]
We introduce the $q$-mean defined 
as a function $\Theta_q:\R^+ \times \R^+ \rightarrow \R^+$ such that
\[
\Theta_q(x,y) = \frac{q-1}{q}\frac{x^q-y^q}{x^{q-1}-y^{q-1}}.
\]
Note that $\Theta_2(x,y)$ is the arithmetic mean. 
Now, the above expression can be split into two factors, $\II^n = \II_1^n+\II_2^n$, defined as
\[
\II_1^n = \frac{q-1}{q}\frac{k}{2}\sum_K\sum_{L\sim K} |K\edge L|u_{KL}^n((\theta_K^{n+1})^q-(\theta_L^{n+1})^q),
\]
\[
\II_2^n = \frac{k}{2}\sum_K\sum_{L\sim K} |K\edge L|(\Theta_2-\Theta_q)(\theta_K^{n+1},\theta_L^{n+1})((\theta_K^{n+1})^{q-1}-(\theta_L^{n+1})^{q-1}).
\]
On the one hand, for the first addend we can symmetrize again such that
\[
\II_1^n = \frac{q-1}{q}k\sum_K (\theta_K^{n+1})^q\sum_{K\sim L}|K\edge L|u_{KL}^n = \frac{q-1}{q}k\sum_K (\theta_K^{n+1})^q(\nabla\cdot u)_K^n \geq - \frac{q-1}{q}\lambda^n\|\theta_{k,h}(t^{n+1})\|_{L^q}^q
\]
where $\lambda^n = k\|(\nabla\cdot u(t^n))_{k,h}^-\|_{L^\infty}$. On the other hand, we estimate $\II_2^n$ using the following bound
\begin{equation}\label{eq:q-mean_estimate}
	|\Theta_2(x,y)-\Theta_q(x,y)| \leq \frac{|q-2|}{q}\frac{|x-y|}{2}
\end{equation}
for all $x,y>0$. More information about the $q$-mean and a detailed proof of the latter estimate can be found on \cite[Appendix A]{SchlichtingSeis18}.
By the estimate \eqref{eq:q-mean_estimate} follows
\[
\II_2^n \geq -\frac{k}{2}\frac{|q-2|}{q}\sum_K\sum_{L\sim K} |K\edge L||u_{KL}^n|\frac{\theta_K^{n+1}-\theta_L^{n+1}}{2}((\theta_K^{n+1})^{q-1}-(\theta_L^{n+1})^{q-1}).
\]
Analogously, for both $\III^n$ and $\IV^n$ the symmetrization procedure might be applied to get the bounds
\[
\III^n \geq \frac{k}{2}\sum_K\sum_{L\sim K} |K\edge L| |u_{KL}^n|\frac{\theta_K^{n+1}-\theta_L^{n+1}}{2}((\theta_K^{n+1})^{q-1}-(\theta_L^{n+1})^{q-1}).
\]
\[
\IV^n = \kappa\frac{k}{2}\sum_K\sum_{L\sim K} |K\edge L|\frac{\theta_K^{n+1}-\theta_L^{n+1}}{d_{KL}}((\theta_K^{n+1})^{q-1}-(\theta_L^{n+1})^{q-1}).
\]
All in all we get the estimate
\begin{equation}\label{eq:first_stab_est}
\begin{aligned}
\|\theta_{k,h}^{n+1}\|_{L^q}^q & + \frac{k}{2}\left(1-\frac{|q-2|}{q}\right)\sum_K\sum_{L\sim K} |K\edge L||u_{KL}^n|\frac{\theta_K^{n+1}-\theta_L^{n+1}}{2}((\theta_K^{n+1})^{q-1}-(\theta_L^{n+1})^{q-1}) \\
& + \kappa\frac{k}{2}\sum_K\sum_{L\sim K} |K\edge L|\frac{\theta_K^{n+1}-\theta_L^{n+1}}{d_{KL}}((\theta_K^{n+1})^{q-1}-(\theta_L^{n+1})^{q-1}) \\
& \leq \|\theta_{k,h}^n\|_{L^q}\|\theta_{k,h}^{n+1}\|_{L^q}^{q-1} + \frac{q-1}{q}\lambda^n\|\theta_{k,h}^{n+1}\|_{L^q}^q.
\end{aligned}
\end{equation}
In order to obtain the first stability estimate \eqref{eq:stability_estimate1} we can drop the second and third addends in~\eqref{eq:first_stab_est} such that, dividing by $\|\theta_{k,h}^{n+1}\|_{L^q}^{q-1}$, we get
\[
\left(1-\frac{q-1}{q}\lambda^n\right)\|\theta_{k,h}^{n+1}\|_{L^q} \leq \|\theta_{k,h}^n\|_{L^q}.
\]
Now, if $k\leq k_{\max}(\alpha)$ it holds that
\[
\frac{q-1}{q}\lambda^n \leq \frac{\alpha-1}{\alpha},
\]
and therefore
\[
\frac{1}{1-\frac{q-1}{q}\lambda^n} \leq 1+\alpha\frac{q-1}{q}\lambda^n \leq \exp\left(\alpha\frac{q-1}{q}\lambda^n\right).
\]
By an iterative argument we get
\[
\|\theta_{k,h}^{n}\|_{L^q} \leq \exp\left(\alpha\frac{q-1}{q}k\sum_{i=1}^{n}\|(\nabla\cdot u(t^{i-1}))_{k,h}^-\|_{L^\infty}\right)\|\theta_h^0\|_{L^q}
\]
for every $n\in \llb 0,N \rrb$ and thus we get the first stability estimate \eqref{eq:stability_estimate1}.

To establish the temporal and spatial gradient estimate \eqref{eq:stability_estimate2} we repeat a similar computation. However now we need to develop a different bound for the term $\I^n$ and thus we use the estimate
\[
rx^{r-1}(x-y) \geq x^r-y^r + \frac{r(r-1)}{2^{3-r}}\left(\frac{x+y}{2}\right)^{r-2}(x-y)^2
\]
that holds for $r\in (1,2]$ and comes from the convexity of the map $x\mapsto x^r$. Then, by setting $x = \theta_K^{n+1}$, $y = \theta_K^n$ and $r\in(1,\min\{q,2\}]$, we get the following lower bound for $\I^n$,
\[
\begin{aligned}
\I^n & = \sum_K |K|(\theta_K^{n+1})^r - \sum_K |K|\theta_K^n (\theta_K^{n+1})^{r-1} = \sum_K |K| (\theta_K^{n+1})^{r-1}(\theta_K^{n+1}-\theta_K^n)  \\
& \geq \frac{1}{r}\sum_K |K| ((\theta_K^{n+1})^r-(\theta_K^n)^r) + \frac{r-1}{2^{3-r}}\sum_K|K|\left(\frac{\theta_K^{n+1}+\theta_K^n}{2}\right)^{r-2}(\theta_K^{n+1}-\theta_K^n)^2,
\end{aligned}
\]
and adding it to the stability estimate \eqref{eq:first_stab_est} instead of the previous one, we get
\[
\begin{aligned}
\frac{1}{r} \|\theta_{k,h}^{n+1}\|_{L^r}^r & + \frac{r-1}{2^{3-r}}\sum_K|K| \left(\frac{\theta_K^{n+1}+\theta_K^n}{2}\right)^{r-2}(\theta_K^{n+1}-\theta_K^n)^2 \\
& +\frac{r-1}{r}\frac{k}{2}\sum_K\sum_{L\sim K} |K\edge L||u_{KL}^n|(\theta_K^{n+1}-\theta_L^{n+1})((\theta_K^{n+1})^{r-1}-(\theta_L^{n+1})^{r-1}) \\
& +\kappa\frac{k}{2}\sum_K\sum_{L\sim K} |K\edge L|\frac{\theta_K^{n+1}-\theta_L^{n+1}}{d_{KL}}((\theta_K^{n+1})^{r-1}-(\theta_L^{n+1})^{r-1}) \\
& \leq \frac{1}{r} \|\theta_{k,h}^n\|_{L^r}^r + \frac{r-1}{r}\lambda^n\|\theta_{k,h}^{n+1}\|_{L^r}^r.
\end{aligned}
\]
We now rewrite the advection and diffusion terms using the following elementary inequality that holds for any $r\in (1,2]$ and $x,y>0$,
\[
(x-y)^2\left(\frac{x+y}{2}\right)^{r-2} \leq (x-y)\frac{x^{r-1}-y^{r-1}}{r-1}.
\]
Choosing $x=\theta_K^{n+1}$, $y=\theta_L^{n+1}$ we thus get
\[
\begin{aligned}
\frac{r-1}{2^{3-r}} & \sum_K|K| \left(\frac{\theta_K^{n+1}+\theta_K^n}{2}\right)^{r-2}(\theta_K^{n+1}-\theta_K^n)^2 \\
& +\frac{r-1}{r}\frac{k}{2}\sum_K\sum_{L\sim K} |K\edge L|\left(|u_{KL}^n|+\frac{\kappa}{d_{KL}}\right)\left(\frac{\theta_K^{n+1}-\theta_L^{n+1}}{2}\right)^{r-2}(\theta_K^{n+1}-\theta_L^{n+1})^2 \\
& \leq \frac{1}{r} \|\theta_{k,h}^n\|_{L^r}^r + \frac{r-1}{r}\lambda^n\|\theta_{k,h}^{n+1}\|_{L^r}^r.
\end{aligned}
\]
Summing over $n$ and applying \eqref{eq:stability_estimate1} it yields the desired estimate \eqref{eq:stability_estimate2} with constant
\[
C_r = 2\frac{\max\{2^{2-r},r\}}{r(r-1)}. \qedhere
\]
\end{proof}

Lemma \ref{lemma:stability_estimate} provides a discrete version of the standard stability and energy estimates in the continuous setting. On the one hand \eqref{eq:stability_estimate1} is the discrete version of \eqref{eq:cont_stability1}, while on the other hand~\eqref{eq:cont_stability2} is reproduced in the numerical scheme setting by \eqref{eq:stability_estimate2} dropping the addends related to the time derivative and the advection field.

A direct consequence of Lemma \ref{lemma:stability_estimate} together with \eqref{eq:bound_datum} and \eqref{eq:bound_field} is that the expressions on the right hand side innu \eqref{eq:stability_estimate1} and \eqref{eq:stability_estimate2} are controlled by $\|\theta^0\|_{L^q}$ and $\|(\nabla\cdot u)^-\|_{L^1(L^\infty)}$ and therefore they are $\BigO(1)$. In particular it holds
\[
\|\theta_{k,h}\|_{L^\infty(L^q)} \lesssim 1,
\]
which is certainly not surprising since that also holds for the exact solutions of \eqref{eq:advdiff}.

Let us introduce now two weak BV estimates which will be a key tool to obtain the desired result from Theorem \ref{theorem}. These estimates are a consequence of numerical diffusion.

\begin{lemma}[BV~estimates]
Let $\theta_{k,h}$ be a solution of the numerical scheme \eqref{eq:scheme1}. Under the assumptions of Theorem \ref{theorem} we get the following BV~estimates
\begin{equation}\label{eq:weakBVtime}
\sum_{n}\sum_K |K||\theta_K^{n+1}-\theta_K^n| \lesssim \sqrt{\frac{T}{k}},
\end{equation}
\begin{equation}\label{eq:weakstrongBVspace}
	k\sum_n\sum_K\sum_{L\sim K} |K\edge L| \bra*{\sqrt{\frac{\kappa}{T}} + \sqrt{\frac{h}{T \norm{u}_\infty}} |u_{KL}^n| } |\theta_K^{n+1}-\theta_L^{n+1}| \lesssim 1 . 
\end{equation}
\end{lemma}

The first estimate on the time discretization \eqref{eq:weakBVtime} does not have a counterpart in the continuous setting and it is a by-product of the numerical diffusion introduced by the temporal discretization of the scheme. The second one \eqref{eq:weakstrongBVspace} instead presents two differentiated parts. First we obtain a spatial \emph{strong} BV estimate
\begin{equation}
\label{eq:strongBVspace}
k\sum_n\sum_K\sum_{L\sim K} |K\edge L||\theta_K^{n+1}-\theta_L^{n+1}| \lesssim \sqrt{\frac{T}{\kappa}} ,
\end{equation}
which is precisely the responsible for carrying an upgrade on the convergence rate from $\BigO(h^{1/2})$ to $\BigO(h)$ in comparison with the transport equation without diffusion,~\cite{SchlichtingSeis17,SchlichtingSeis18}. This BV estimate can be understood as the discrete analogous to 
\[
\|\nabla\theta\|_{L^1(L^1)} \lesssim \sqrt{\frac{T}{\kappa}}.
\]
Then we also obtain the spatial \emph{weak} BV estimate
\begin{equation}
\label{eq:weakBVspace}
k\sum_n\sum_K\sum_{L\sim K} |K\edge L||u_{KL}^n||\theta_K^{n+1}-\theta_L^{n+1}| \lesssim \sqrt{\frac{T\|u\|_\infty}{h}}
\end{equation}
that is a consequence of the numerical diffusion introduced by the spatial discretization and can be read as the surviving part in the limit $\kappa\to 0$. It is precisely the weak BV estimate obtained in \cite[Proposition 1]{SchlichtingSeis18} for the transport equation.
\begin{proof}
We start proving \eqref{eq:weakBVtime}. Let us first consider a nonnegative initial datum. Let $r\in(1,\min\{2,q\}]$ and smuggle into \eqref{eq:weakBVtime} the weight $((\theta_K^{n+1}+\theta_K^n)/2)^{(r-2)/2}$ such that
\[
\sum_K |K||\theta_K^{n+1}-\theta_K^n|\left(\frac{\theta_K^{n+1}+\theta_K^n}{2}\right)^{\frac{r-2}{2}}\left(\frac{\theta_K^{n+1}+\theta_K^n}{2}\right)^{\frac{2-r}{2}} = \I^n.
\]
Then, via Cauchy-Schwarz,
\[
\I^n \leq \left[\sum_K|K|(\theta_K^{n+1}-\theta_K^n)^2\left(\frac{\theta_K^{n+1}+\theta_K^n}{2}\right)^{r-2}\right]^{1/2}\left[\sum_K|K|\left(\frac{\theta_K^{n+1}+\theta_K^n}{2}\right)^{2-r}\right]^{1/2}.
\]
By Lemma \ref{lemma:stability_estimate} and \eqref{eq:bound_field}, the first factor of the product is controlled by a constant depending on $r$, the $L^1(L^\infty)$ norm of $(\nabla\cdot u)^-$ and the $L^r$ norm of the initial datum. Therefore, summing over $n$ and applying Jensen's inequality for the time variable now we can write,
\[
\begin{aligned}
\sum_n\I^n & \lesssim \sum_n \left[\sum_K|K|\left(\frac{\theta_K^{n+1}+\theta_K^n}{2}\right)^{2-r}\right]^{1/2} 
\!\!\!\!\leq \sum_n \left[\sum_K|K|((\theta_K^{n+1})^{2-r}+(\theta_K^n)^{2-r})\right]^{1/2} \\
& \leq 2T^{1/2} \left(\sum_n \|\theta_{k,h}(t^n)\|_{L^{2-r}}^{2-r}\right)^{1/2} \lesssim \sqrt{\frac{T}{k}}\|\theta_{k,h}\|_{L^1(L^{2-r})}^{(2-r)/2}.
\end{aligned}
\]
Hence, by \eqref{eq:bound_datum} and \eqref{eq:bound_field} we get the weak BV estimate \eqref{eq:weakBVtime} for nonnegative initial data. Once this is established, for general initial data the estimate follows via triangle inequality.

We argue analogously to get the estimate \eqref{eq:weakstrongBVspace}. We will obtain the two estimates \eqref{eq:strongBVspace} and \eqref{eq:weakBVspace} separately. Let us start with the strong BV estimate \eqref{eq:strongBVspace}. Consider a non negative initial datum since for a general case we can just apply a triangle inequality. Smuggling the same weight as before, with $r\in(1,\min\{q,2\}]$, together with a factor $d_{KL}$ we can write via Cauchy-Schwarz inequality,
\[
k \sum_n \sum_K\sum_{L\sim K} |K\edge L||\theta_K^{n+1}-\theta_L^{n+1}| = k\sum_n (\II_S^n)^{1/2}(\III_S^n)^{1/2} = \bra*{k\sum_n \II_S^n}^{1/2}\bra*{ k \sum_n \III_S^n}^{1/2} 
\]
with
\begin{align*}
\II_S^n &= \sum_K\sum_{L\sim K}|K\edge L|\frac{(\theta_K^{n+1}-\theta_L^{n+1})^2}{d_{KL}}\left(\frac{\theta_K^{n+1}+\theta_L^{n+1}}{2}\right)^{r-2}, \\
\III_S^n &= \sum_K\sum_{L\sim K}|K\edge L|d_{KL}\left(\frac{\theta_K^{n+1}+\theta_L^{n+1}}{2}\right)^{2-r}.
\end{align*}
The term involving $\II_S^n$ is controlled thanks to \eqref{eq:stability_estimate2} by
\[
\left(k\sum_n \II_S^n\right)^{1/2} \lesssim \frac{1}{\sqrt{\kappa}}.
\]
For $\III_S^n$ we can use the identity $((x+y)/2)^{2-r}\leq x^{2-r}+y^{2-r}$ for any $x,y>0$ and the trivial bound $d_{KL}\leq 2h$. Then by the isoperimetric property of the mesh \eqref{eq:isoperimetric} we get
\[
\III_S^n \leq h\sum_K (\theta_K^{n+1})^{2-r}\sum_{L\sim K}|K\edge L| \lesssim \sum_K |K| (\theta_K^{n+1})^{2-r} = \|\theta_{k,h}(t^n)\|_{L^{2-r}}.
\]
Again we can estimate $\|\theta_{k,h}(t^n)\|_{L^{2-r}}$ by $\|\theta_{k,h}(t^n)\|_{L^{r}}$ and a factor depending on $|\Omega|$ so that it yields the remaining term
\[
\left(k\sum_n \III_S^n\right)^{1/2} \lesssim \sqrt{T}.
\]
Thus, we obtain the strong BV estimate \eqref{eq:strongBVspace}.

For the weak BV estimate \eqref{eq:weakBVspace} we follow a similar argument. Let $r\in(1,\min\{q,2\}]$ and apply Cauchy-Schwarz as before,
\[
k \sum_n \sum_K\sum_{L\sim K} |K\edge L||u_{KL}^n||\theta_K^{n+1}-\theta_L^{n+1}| = k\sum_n (\II_W^n)^{1/2}(\III_W^n)^{1/2} = \bra*{k\sum_n \II_W^n}^{\!1/2}\bra*{ k \sum_n \III_W^n}^{\!1/2}
\]
where now we define
\begin{align*}
\II_W^n &= \sum_K\sum_{L\sim K}|K\edge L||u_{KL}^n|^2(\theta_K^{n+1}-\theta_L^{n+1})^2\left(\frac{\theta_K^{n+1}+\theta_L^{n+1}}{2}\right)^{r-2}, \\
\III_W^n &= \sum_K\sum_{L\sim K}|K\edge L|\left(\frac{\theta_K^{n+1}+\theta_L^{n+1}}{2}\right)^{2-r}.
\end{align*}
Then a direct application of \eqref{eq:stability_estimate2} and following the previous argument for the strong BV estimate we obtain
\[
\left(k\sum_n \II_W^n\right)^{1/2} \lesssim \sqrt{\|u\|_\infty}\qquad \text{and}\qquad
\left(k\sum_n \III_W^n\right)^{1/2} \lesssim \sqrt{\frac{T}{h}}
\]
so that we complete the proof of \eqref{eq:weakstrongBVspace}.
\end{proof}

\section{Logarithmic Kantorovich--Rubinstein distances}\label{S:optimal_transport}

In this section we give an overview about the optimal transport distance that we use to measure the errors in Theorem \ref{theorem}. For a deeper understanding and proofs of the results here mentioned we refer to the monograph by Villani \cite{Villani03}.

Consider two nonnegative measures $\mu_1,\mu_2\in L^1_+(\Omega) = \{\mu\in L^1(\Omega)\mid \mu\geq 0\}$.  We define the set of all transport plans between $\mu_1$ and $\mu_2$, denoted by $\Pi(\mu_1,\mu_2)$, as the collection of all measures $\pi$ on $\Omega\times\Omega$ such that
\[
\pi[A\times\Omega] = \mu_1[A] \text{ and } \pi[\Omega\times A] = \mu_2[A] \quad \text{for all measurable } A\subseteq \Omega,
\]
or equivalently such that
\[
\int_{\Omega\times\Omega}(f_1(x)+f_2(y))d\pi(x,y) = \io f_1d\mu_1 + \io f_2d\mu_2
\]
for all $f_1\in L^1(\mu_1)$, $f_2\in L^1(\mu_2)$.

We define a nondecresing function $c:[0,\infty)\rightarrow[0,\infty)$ called \emph{cost function} that models the cost of the transpost of an infinitesimal part of the configurations. The optimal transport problem consists of finding the transport plan in $\Pi(\mu_1,\mu_2)$ that minimizes the total cost of transportation from one configuration to another, more precisely
\[
\D_c(\mu_1,\mu_2) = \inf_{\pi\in\Pi(\theta_1,\theta_2)}\iint_{\Omega\times\Omega} c(|x-y|)d\pi(x,y).
\]

When the cost function is given by a distance $d(x,y)=c(|x-y|)$ then the quantity $\D_c(\mu_1,\mu_2)$ defines a metric in the space of measures, the so-called \emph{Kantorovich--Rubinstein metric}. In addition, if the cost function is concave, the optimal transport problem admits a dual formulation that reads as follows,
\[
\D_c(\mu_1,\mu_2) = \sup_{\zeta:\Omega\rightarrow\R}\left\lbrace \io \zeta(x)(\mu_1(x)-\mu_2(x))dx : |\zeta(x)-\zeta(y)|\leq c(|x-y|) \right\rbrace,
\]
where the optimal $\zeta$ in this representation will be referred as the \emph{Kantorovich potential}. 

By means of the dual formulation we see that the distance $\D_c(\mu_1,\mu_2)$ only depends on the difference $\mu_1-\mu_2$ (when the cost function is concave) and we thus can consider negative or not-signed densities as long as both $\mu_1$ and $\mu_2$ are of the same mass, i.e. $\mu_1[\Omega] = \mu_2[\Omega]$. This way, the quantity $\D_c(\mu_1,\mu_2)$ defines a distance on the space of densities with same total mass.

In this context, for any mean-zero density, i.e. for any $\mu\in L^1(\Omega)$ such that 
\[
\io \mu(x)dx = 0
\]
we can conveniently define the norm
\[
\D_c(\mu) = \D_c(\mu,0) = \D_c(\mu^+,\mu^-).
\]

For the results here presented we are specifically interested in a particular concave cost function. For any $\delta>0$ we define
\[
c(z) = \log\left(\frac{z}{\delta}+1\right)
\]
such that the optimal transportation distance that we use in Theorem \ref{theorem} reads,
\begin{equation}\label{eq:optimal_transport_distance}
\D_\delta(\mu_1,\mu_2) = \inf_{\pi\in\Pi(\theta_1,\theta_2)}\iint_{\Omega\times\Omega} \log\left(\frac{|x-y|}{\delta}+1\right)d\pi(x,y)
\end{equation}
for all $\mu_1,\mu_2\in L^1(\Omega)$ with same total mass. Furthermore, the Kantorovich potential associated to this logarithmic cost has the Lipschitz property,
\begin{equation}\label{eq:dzeta}
\|\nabla\zeta\|_{L^\infty} \leq \frac{1}{\delta}.
\end{equation}

The logarithmic cost has been used in previous works to study transport and advection-diffusion equations, see \cite{CrippaNobiliSeisSpirito17,NavarroFernandezSchlichtingSeis21,Seis17}, since similar expressions appear naturally when searching for stability estimates for the transport equation in the smooth setting. It is of particular interest for us here because the distance is singular when $\delta\rightarrow 0$, therefore if we find a uniform bound for the distance as $\delta\rightarrow 0$, it means that $\mu_1\rightarrow\mu_2$ (in some sense) with rate, at most, $\delta$. Since finding optimal rates of convergence is the main goal of this work, it appears natural to make use of the distance \eqref{eq:optimal_transport_distance}.

The convergence $\mu_1\rightarrow \mu_2$ takes place in the weak topology, i.e. $\mu_1\rightharpoonup \mu_2$, because one of the most powerful properties of the Kantorovich--Rubinstein distances is that they metrize weak convergence, see \cite[Theorem 7.12]{Villani03}. This means that $\D_c(\mu_1,\mu_2)\rightarrow 0$ if and only if $\mu_1\rightharpoonup \mu_2$.

To sum up we present a result that deals with the differentiability properties of the properties of the logarithmic Kantorovich--Rubinstein distance \eqref{eq:optimal_transport_distance}. This is very useful for the derivation of some of the stability estimates presented on the next section, especially when measuring the distance between two solutions of the advection-diffusion equation \eqref{eq:advdiff}.
\begin{lemma}
Let $\mu_1,\mu_2\in L^\infty(L^q)\cap L^1(W^{1,1})$ with $q>1$ be two solutions to the advection-diffusion equation \eqref{eq:advdiff}  with vector fields $u_1$, $u_2$ and diffusion coefficients $\kappa_1$, $\kappa_2$ respectively. Then the mapping $t\mapsto\D_\delta(\mu_1(t),\mu_2(t))$ is absolutely continuous with
\begin{equation}\label{eq:dtD}
\begin{split}
\frac{d}{dt}\D_\delta(\mu_1(t),\mu_2(t))& = \int_\Omega \nabla\zeta_t\cdot (u_1(t)\mu_1(t)-u_2(t)\mu_2(t))dx \\
& \qquad - \int_\Omega \nabla\zeta_t\cdot(\kappa_1\nabla\mu_1(t)-\kappa_2\nabla\mu_2(t))dx,
\end{split}
\end{equation}
where $\zeta_t$ is the Kantorovich potential corresponding to $\D_\delta(\mu_1(t),\mu_2(t))$.
\end{lemma}
For a proof of this result in the whole space we refer to \cite[Lemma 4]{NavarroFernandezSchlichtingSeis21}. Everything is straightforwardly adaptable to bounded domains with the no-flux boundary condition \eqref{eq:noflow}. Furthermore, for a bounded domain $\Omega$, if $q>1$ then the first moments are finite and standard embeddings of Lebesgue spaces imply $\theta_1,\theta_2\in L^1(W^{1,1})$.

\section{Proof of Theorem \ref{theorem}}\label{S:proof}

In this section we will prove the main result of the paper. In order to do so we need to derive all the error estimates coming from the different discretizations that contribute to the stability estimate \eqref{eq:principal_estimate}. There are two main sources of error: on the one hand the discretization in time and space of the initial datum and the vector field and on the other hand there is the error associated to the scheme, also known as truncation error. For the diffusionless transport equation one can see (for instance, in \cite{SchlichtingSeis17,SchlichtingSeis18}) that the error that governs the convergence of the numerical solution comes exclusively in form of truncation error. However in our case we will see how both sources of error, truncation and discretization of data, contribute equally to the final estimate.

Before turning to the proof of the Theorem let us first mention essential mathematical tools to study stability estimates for the advection-diffusion equations in a low regularity framework.

On the one hand, we use the Hardy-Littlewood maximal function from the Calderón-Zygmund theory in harmonic analysis. Given a measurable function $f:\R^d\rightarrow\R$, we say $M$ is the maximal function operator and it is defined by
\[
Mf(x) = \sup_{R>0} \frac{1}{R^d}\int_{B_R(x)\cap\Omega} |f(y)|dy.
\]
The operator is continuous from $L^p$ to $L^p$ for every $1<p\leq \infty$ and therefore we get the estimate,
\begin{equation}\label{eq:morrey}
	\|Mf\|_{L^p} \lesssim \|f\|_{L^p}, \quad \text{for } 1<p\leq\infty.
\end{equation}
Moreover, via the maximal function we can establish bounds for the different quotients of a measurable function through the so-called \emph{Morrey's inequality}, that is
\begin{equation}\label{eq:morrey_estimate}
	\frac{|f(x)-f(y)|}{|x-y|} \lesssim (M\nabla\overline{f})(x) + (M\nabla\overline{f})(y)
\end{equation}
for almost every $x,y\in\Omega$ and where $\overline{f}$ denotes a Sobolev regular extension of $f$ to the full space $\R^d$. These type of arguments with Morrey's inequality to deal with transport and advection-diffusion equations on the DiPerna-Lions setting have been broadly used, see for instance \cite{BouchutCrippa13,CrippaNobiliSeisSpirito17,NavarroFernandezSchlichtingSeis21,Seis17}.

On the other hand, as stated in the introduction the stability estimate~\eqref{eq:cont_stability2} provides for any $I\subset [0,T]$ an explicit control on the $L^1(I;W^{1,1}(\Omega))$ norm of the solution to \eqref{eq:advdiff}. We can see this by choosing $r\in(1,\min\{q,2\}]$, then we have
\begin{equation}\label{eq:L1estimate}
	\int_I \int_\Omega |\nabla\theta|dx\,dt
	\leq \biggl(\int_I \int_\Omega |\theta|^{r-2}|\nabla\theta|^2dx\,dt\biggr)^{1/2}
	\biggl(\int_I \int_\Omega |\theta|^{2-r}dx\,dt\biggr)^{1/2} \lesssim \sqrt{\frac{\abs{I}}{\kappa}}\|\theta^0\|_{L^r}
\end{equation}
where we have used Hölder's inequality and we have estimated $\|\theta\|_{L^\infty(L^{2-r})}$ by $\|\theta\|_{L^\infty(L^r)}$ with a factor depending on $|\Omega|$ via
\[
\int_I \int_\Omega |\theta|^{2-r}dx\,dt \lesssim \int_I \left(\int_\Omega |\theta|^rdx\right)^{\frac{2-r}{r}}dt \lesssim \abs{I}\|\theta\|_{L^\infty(L^r)}^{2-r} \lesssim \abs{I}\|\theta^0\|_{L^r}^{2-r}.
\]

\subsection{Error due to the discretization of the data}\label{ss:error_discretization}

We start with the contribution to the error estimates caused by the discretization in time. 

\begin{lemma}\label{lemma:disc_time}
Let $t\in [t^n,t^{n+1})$ with $n\in \llb 0,N-1\rrb$. Then it holds
\begin{equation}\label{eq:lemma_disc_time}
\D_\delta (\theta(t),\theta(t^n)) \lesssim \frac{k\|u\|_\infty+\sqrt{k\kappa}}{\delta}.
\end{equation}
\end{lemma}
\begin{proof}
Let $\zeta_t$ be the optimal Kantorovich potential corresponding to the distance $\D_\delta (\theta(t),\theta(t^n))$ at time $t\in [t^n,t^{n+1})$ for some $n\in \llb 0,N-1\rrb$, such that
\[
\D_\delta (\theta(t),\theta(t^n)) = \int_\Omega \zeta_t(x)(\theta(t,x)-\theta(t^n,x))dx.
\]
By means of \eqref{eq:dtD} we can rewrite the distance as
\[
\D_\delta (\theta(t),\theta(t^n)) = \int_{t^n}^t\int_\Omega \nabla\zeta_t(x)\cdot u(s,x)\theta(s,x)dxds - \kappa\int_{t^n}^t\int_\Omega \nabla\zeta_t(x)\cdot\nabla\theta(s,x)dxds,
\]
that to shorten the notation we denote as $\D_\delta (\theta(t),\theta(t^n)) = \I+\II$. The first addend can be controlled by the standard estimate \eqref{eq:dzeta} as follows,
\[
\I = \int_{t^n}^t\int_\Omega \nabla\zeta_t(x)\cdot u(s,x)\theta(s,x)dxds \lesssim \frac{k}{\delta}\|u\|_\infty\|\theta\|_{L^\infty(L^1)}.
\]
For the second term, we apply the estimate~\eqref{eq:L1estimate} on the time interval $[t^n, t)$ and we get
\[
\|\nabla\theta\|_{L^1([t^n,t);L^1(\Omega))}\lesssim \sqrt{\frac{t-t^n}{\kappa}}\leq \sqrt{\frac{k}{\kappa}}
\]
and thus it yields the bound for $\II$ via
\[
\II = -\kappa\int_{t^n}^t\int_\Omega \nabla\zeta_t(x)\cdot\nabla\theta(s,x)dxds \leq \frac{\kappa}{\delta} \int_{t^n}^t\int_\Omega |\nabla\theta(s,x)|dxds\lesssim \frac{\sqrt{k\kappa}}{\delta}.
\]
Thus, putting everything together it yields the estimate \eqref{eq:lemma_disc_time}.
\end{proof}
Next in order we study the error caused by the spatial discretization of the initial datum $\theta^0$. We define $\theta_h^0(x) = \theta_K^0(x)$ as in \eqref{eq:datum_discrete} piecewise for almost every $x\in K$ and for each $K\in \cT$. This result is a straightforward consequence of the stability estimate for the advection-diffusion equation \eqref{eq:cont_stab_est}.
\begin{lemma}\label{lemma:disc_datum}
Let $\theta^h$ be the solution to the advection-diffusion equation \eqref{eq:advdiff} with initial datum $\theta_h^0$. Then it holds
\begin{equation}\label{eq:lemma_disc_datum}
\sup_{0\leq t\leq T} \D_\delta(\theta(t),\theta^h(t)) \lesssim 1+\frac{h}{\delta}.
\end{equation}
\end{lemma}
\begin{proof}
In this case $\theta$ and $\theta^h$ are solutions to the same equation with same velocity fields and same diffusion coefficients, therefore a direct application of \eqref{eq:cont_stab_est} yields
\[
\sup_{0\leq t\leq T} \D_\delta(\theta(t),\theta^h(t)) \lesssim 1 + \D(\theta^0,\theta^0_h).
\]
Now let us write $\zeta_t$ to denote the optimal Kantorovich potential such that it holds
\[
\D(\theta^0,\theta^0_h) = \int_\Omega \zeta_t(x)(\theta^0(x)-\theta_h^0)dx = \int_\Omega (\zeta_t(x)-(\zeta_t)_h(x))\theta^0(x)dx
\]
where the second equality comes from the symmetry property of the cell-averaging $(\cdot)_h$ operator, that is
\[
\int_\Omega f(x)g_h(x)dx 
= \sum_K |K| \avint_K\avint_K f(x)g(y)dy\,dx = \int_\Omega f_h(x)g(x)dx
\]
for all integrable $f$ and $g$ such that its product is also integrable. Furthermore, we use the definition of the Kantorovich potential together with its Lipschitz bound \eqref{eq:dzeta} pointwise in $x\in K$ so that,
\[
|\zeta_t(x)-(\zeta_t)_h(x)| \leq \avint_K |\zeta_t(x)-\zeta_t(y)| dy \leq \avint_K\log\bra*{\frac{|x-y|}{\delta}+1} dy \leq \log\bra*{\frac{h}{\delta}+1} \leq \frac{h}{\delta}.
\]
We thus find the final estimate \eqref{eq:lemma_disc_datum} just by combining everything.
\end{proof}

In addition we must also consider the error due to the time discretization for the coefficients of the equation. We denote by $u^k$ the vector field averaged in time over $[t^n,t^{n+1})$ as follows,
\[
u^k(t,x) = \avint_{t^n}^{t^{n+1}} u(t,x)dt \quad \text{for a.e. } t\in[t^n,t^{n+1}).
\]

\begin{lemma}\label{lemma:disc_vectorfield}
Let $\theta^k$ be the solution to the advection-diffusion equation \eqref{eq:advdiff} with vector field $u^k$. Then it holds for any $m\in\llb 0,N \rrb$
\begin{equation}\label{eq:lemma_disc_vectorfield}
\D(\theta(t^m,\cdot),\theta^k(t^m,\cdot)) \lesssim 1+\frac{k(\|u\|_\infty+1)}{\delta} + \frac{\sqrt{k\kappa}}{\delta}.
\end{equation}
\end{lemma}
For the proof of the Lemma we need to introduce a stochastic Lagrangian representation for the advection-diffusion equation. Consider a filtered probability space $(U, \cF, \cF_t , \Prob)$, for any $t\geq 0$ we say the map $X_t:\Omega\rightarrow\Omega$ is an \emph{stochastic Lagrangian flow} if for every $x\in\Omega$ it solves the stochastic differential equation
\begin{equation}\label{eq:SDE_Flow}
X_t = X_t^x  = x+\int_0^t u(s,X_s(x))ds + \sqrt{2\kappa}\,B_t - \int_0^t n(X_s(x))dL_s.
\end{equation}
Here $\set{B_t}_{t\geq 0}$ is a $\cF_t$-adapted Brownian motion and $\set{L_t}_{t\geq 0}$ is an $\cF_t$-adapted local time of the process $\set{X_t}_{t\geq 0}$ at the boundary $\partial \Omega$. By the classic Doob maximal martingale inequality (see~\cite{RevuzYor1999}), we have for any $q>1$ the bound for the Brownian motion,
\begin{equation}\label{eq:Doob}
	\EX\pra*{ \sup_{0\leq s \leq t}\abs*{B_s}^q}^{\frac{1}{q}} \leq \frac{q}{q-1} \sqrt{t} . 
\end{equation}
Since the setting and tools needed for the proof of this Lemma use some language from stochastic analysis and differs from the rest of the mathematical tools presented in this paper, we include for the convenience of the reader the Appendix \ref{s:SLF} reviewing some of the abstract setting and the formal definitions that will be used along this proof.

\begin{proof}[Proof of Lemma \ref{lemma:disc_vectorfield}]
We can assume by a density argument that $u$ and $u^k$ are smooth in space and continuous in time. Indeed, this a consequence of a classic approximation argument leading to the emergence of a commutator, which can be estimated along the lines of~\cite[Lemma 2.1]{DiPernaLions89} or \cite[Section 2]{DeLellis2008}, and the fact that the logarithmic Kantorovich--Rubinstein distance metrizes weak convergence.

Without loss of generality, we assume that $\theta^0$ is a probability measure.  Hence, by the results stated in the Appendix we find processes $\{X_t\}_{t\geq 0}$ and $\{X_t^k\}_{t\geq 0}$, strong solutions to the reflected SDE~\eqref{eq:SDE} started with law $\theta^0$ driven by the same Brownian motion $\{B_t\}_{t\geq 0}$ with vector field $u$ and $u^k$, respectively. The according local times at the boundary are denoted by $\set*{L_t}_{t\geq 0}$ and $\{L_t^k\}_{t\geq 0}$. In this way, we constructed a pathwise coupling of $\theta(t)$ and $\theta^k(t)$, i.e. $\operatorname{law} X_t = \theta(t)$ and $\operatorname{law} X_t^k = \theta^k(t)$ and we can straightforwardly estimate the logarithmic Kantorovich--Rubinstein distance with the help of the Lagrangian coupling for any $t\in [0,T]$ by
\begin{align*}
\D_\delta(\theta(t),\theta^k(t)) &\leq  \EX_{\theta^0}\pra*{\log\left(\frac{|X_{t}-X^k_{t}|}{\delta}+1\right)}\\
&\leq e^\frac{t}{r_0}  \EX_{\theta^0}\pra*{\log\left(\frac{|X_{t}-X^k_{t}|}{\delta} e^{-\frac{1}{2r_0}\bra*{L_t+ L_t^k}} +1\right)},\\
&\lesssim  \EX_{\theta^0}\pra*{\log\left(\frac{|X_{t}-X^k_{t}|}{\delta} e^{-\frac{1}{2r_0}\bra*{L_t+ L_t^k}} +1\right)},
\end{align*}
where we used the fact that the boundary local times satisfy $\abs{L_t} , \abs{L_t^k} \leq t$ for any $t\in [0,T]$ and where $r_0$ is the constant given by the uniform exterior ball condition for the domain \eqref{ass:exteriorball}. We have also estimated $e^{t/r_0}$ with $e^{T/r_0}$ and absorbed this constant in $\lesssim$.  Hence, by telescoping and using that $X_0=X_0^k$, we arrive at the estimate
\begin{align*}
	\D_\delta(\theta(t^m),\theta^k(t^m))\lesssim \sum_{n=0}^{m-1}\Biggl(& \EX_{\theta^0}\pra*{\log\left(\frac{|X_{t^{n+1}}-X^k_{t^{n+1}}|}{\delta} e^{-\frac{1}{2r_0}\bra*{L_{t^{n+1}}+ L_{t^{n+1}}^k}}+1\right)} \\
		&- \EX_{\theta^0}\pra*{\log\left(\frac{|X_{t^n}-X^k_{t^n}|}{\delta}e^{-\frac{1}{2r_0}\bra*{L_{t^n}+ L_{t^n}^k}}+1\right)}\Biggr) .
\end{align*}
The representation~\eqref{eq:SDE_Flow} and It\^o's formula allows to estimate for any $n\in \llb 0,m-1 \rrb$
\begin{align*}
	&\EX_{\theta^0}\pra[\Bigg]{\log\left(\frac{|X_{t^{n+1}}-X^k_{t^{n+1}}|}{\delta} e^{-\frac{1}{2r_0}\bra*{L_{t^{n+1}}+ L_{t^{n+1}}^k}}+1\right) - \log\left(\frac{|X_{t^n}-X^k_{t^n}|}{\delta}e^{-\frac{1}{2r_0}\bra*{L_{t^n}+ L_{t^n}^k}}+1\right)}\\
	&= 
	\EX_{\theta^0}\pra[\Bigg]{\int_{t^{n}}^{t^{n+1}} e^{-\frac{1}{2r_0}\bra*{L_{t}+ L_{t}^k}}
		\frac{\frac{X_t - X_t^k}{\abs{X_t - X_t^k}} \cdot \bra*{ dX_t- dX_t^k} - \frac{1}{2r_0} \abs{X_t - X_t^k} \bra*{ dL_t + dL_t^k}}{|X_{t}-X^k_{t}|e^{-\frac{1}{2r_0}\bra*{L_{t}+ L_{t}^k}}+\delta}}\\
	&\lesssim 
	\EX_{\theta^0}\pra[\Bigg]{\int_{t^{n}}^{t^{n+1}}
		\frac{\frac{X_t - X_t^k}{\abs{X_t - X_t^k}} \cdot \bra*{ dX_t- dX_t^k} - \frac{1}{2r_0} \abs{X_t - X_t^k} \bra*{ dL_t + dL_t^k}}{|X_{t}-X^k_{t}|e^{-\frac{1}{2r_0}\bra*{L_{t}+ L_{t}^k}}+\delta}}\\
	&\leq \EX_{\theta^0}\pra[\Bigg]{\int_{t^{n}}^{t^{n+1}} \frac{\abs*{ u(t,X_t) - u^k(t,X_t^k)}}{|X_{t}-X^k_{t}|e^{-\frac{1}{2r_0}\bra*{L_{t}+ L_{t}^k}}+\delta}\, dt} \\
	&\quad - \EX_{\theta^0}\pra[\Bigg]{ \int_{t^{n}}^{t^{n+1}} \frac{\frac{X_t - X_t^k}{\abs{X_t - X_t^k}}\cdot n(X_t) \, dL_t - \frac{X_t - X_t^k}{\abs{X_t - X_t^k}} \cdot n(X_t^k) \, dL_t^k + \frac{1}{2r_0} \abs{X_t - X_t^k} \bra{ dL_t + dL_t^k}}{{|X_{t}-X^k_{t}|e^{-\frac{1}{2r_0}\bra*{L_{t}+ L_{t}^k}}+\delta}}} . %
\end{align*}
Next, we rearrange the integrands in the dominator of the second term in such a way that those have a sign thanks to the exterior ball condition~\eqref{ass:exteriorball}. Indeed, we observe that for $X_t,X_t^k\in \overline \Omega$, one has
\begin{align*}
	\frac{X_t - X_t^k}{\abs{X_t - X_t^k}}\cdot n(X_t) + \frac{1}{2r_0} \abs*{ X_t - X_t^k} \geq 0 ,
\end{align*}
and
\begin{align*}
	- \frac{X_t - X_t^k}{\abs{X_t - X_t^k}}\cdot n(X_t^k) + \frac{1}{2r_0} \abs*{ X_t - X_t^k} =  \frac{X_t^k-X_t}{\abs{X_t^k - X_t}}\cdot n(X_t^k) + \frac{1}{2r_0} \abs*{ X_t^k - X_t}  \geq 0.
\end{align*}
Hence, it is enough to continue to estimate the first one, for which we first get rid of the exponential factor in the denominator again using the property $\abs{L_t} , \abs{L_t^k} \leq t$. Summarizing our findings so far, we get
\begin{equation}\label{eq:intermediate_bound_lemma_stochastic}
	\D_\delta(\theta(t^m),\theta^k(t^m))\lesssim \exp\bra*{\frac{2 t^m}{r_0}} \sum_{n=0}^{m-1} \I^n 
\end{equation}
where
\[
\I^n = \EX_{\theta^0}\pra*{ \int_{t^n}^{t^{n+1}}\frac{\abs*{ u(t,X_t) - u^k(s,X_t^k)} \, ds} {\abs*{X_{t}-X_{t}^k} + \delta}} . 
\]
Using the definition of $u^k$ and Morrey's estimate \eqref{eq:morrey_estimate} we can bound the first addend by
\begin{align*}
	\abs*{ u(s,X_t) - u^k(s,X_t^k)} &\leq \avint_{t^n}^{t^{n+1}} \abs*{ u(t,X_t) - u(t,X_s^k)} \, ds \\
	&\lesssim \avint_{t^n}^{t^{n+1}} \bra*{ (M\nabla\overline u)(t,X_t)+(M\nabla\overline u)(t,X^k_s) }|X_t-X_s^k| \, ds  \,.
\end{align*}
Plugging this estimate into $\I^n$, we introduce the normalized Lebesgue measure 
\[ d\omega_0(x)= \frac{\dsOne_{\Omega}(x)}{|\Omega|} \,dx\] 
and using Hölder's inequality we can write
\[
\begin{aligned}
	\I^n & \lesssim |\Omega| \int_{t^n}^{t^{n+1}}\avint_{t^n}^{t^{n+1}}\EX_{\omega_0}\pra*{\bra*{ (M\nabla\overline u)(s,X_s)+(M\nabla\overline u)(s,X^k_\tau) } \frac{|X_s-X_\tau^k|}{\abs*{X_{t^n}-X_{t^n}^k} + \delta}|\theta^0|}d\tau\, ds \\
	&\lesssim |\Omega| \int_{t^n}^{t^{n+1}}\avint_{t^n}^{t^{n+1}} \EX_{\omega_0}\pra*{ |M\nabla\overline u(s,X_s)|^p +|M\nabla\overline u(s,X_\tau^k)|^p}^\frac{1}{p} \EX_{\omega_0}\pra*{ \left( \frac{|X_s-X_\tau^k|}{\abs*{X_{s}-X_{s}^k} + \delta}|\theta^0| \right)^q }^{\frac{1}{q}}d\tau\, ds \\
	&\leq \int_{t^n}^{t^{n+1}}\avint_{t^n}^{t^{n+1}}\bra*{ \io |M\nabla\overline u(s,x)|^p \, d\omega_s + \io |M\nabla\overline u(s,x)|^p\,d\omega^k_\tau}^\frac{1}{p} \EX_{\theta^0}\pra*{ \left( \frac{|X_s-X_\tau^k|}{\abs*{X_{s}-X_{s}^k} + \delta}|\theta^0| \right)^q }^{\frac{1}{q}},
\end{aligned}
\]
where $\omega_t$ and $\omega_t^k$ are by the representation~\eqref{eq:SLF} solutions to~\eqref{eq:advdiff} with initial datum $\omega_0$ driven by $u$ and $u^k$, respectively.

Now the $L^p$ norm of the maximal function is directly controlled by the fundamental inequality for maximal functions \eqref{eq:morrey}. For the rest we can apply the elemental inequality,
\[
|X_s-X_\tau^k|^q \leq 2^{q-1}(|X_{s}-X^k_{s}|^q+|X^k_{s}-X^k_\tau|^q)
\]
and by the definition of the stochastic flow \eqref{eq:SDE_Flow} we have for any $t,s\in [t^n,t^{n+1})$, $s\leq t$, the estimate
\[
\begin{aligned}
	\EX\pra*{|X_t-X_{s}|^q} & \lesssim \EX\pra*{\abs*{ \int_{t^n}^{t^{n+1}}u(s,X_s)ds }^q} + (2\kappa)^{q/2}\EX\pra*{\abs*{ B_{t^{n+1}}-B_{t^n} }^q} + \EX\pra*{\abs*{ \int_{t^n}^{t^{n+1}}n(X_s)dL_s }^q} \\
	& \leq \|u\|_\infty^qk^q + (2\kappa k)^{q/2} + k^q
\end{aligned}
\]
where we have used the Doob maximal martingale inequality for the Brownian motion \eqref{eq:Doob} and the standard bound for the $\cF_t$-adapted process $L_t$ \eqref{eq:boundary_process_properties} together with the trivial property of the normal vector $\|n\|_{L^\infty}=1$. Therefore, by means of these last two inequalities we can write
\[
\begin{aligned}
\EX\pra*{\left( \frac{|X_s-X_\tau^k|}{\abs*{X_{s}-X_{s}^k} + \delta} \right)^q} & \leq \EX\pra*{ \frac{2^{q-1}(|X_{s}-X^k_{s}|^q+|X^k_{s}-X^k_\tau|^q)}{(\abs*{X_{s}-X_{s}^k} + \delta)^q}} \\
& \lesssim \EX\pra*{ \frac{|X_{s}-X^k_{s}|^q}{\abs*{X_{s}-X_{s}^k}^q}+\frac{|X^k_{s}-X^k_\tau|^q}{\delta^q}} \\
& \lesssim 1 + \frac{(\|u\|_\infty^q+1)k^q+(\kappa k)^{q/2}}{\delta^q}.
\end{aligned}
\]
Finally, noticing that $(1+x^q)^{1/q}\leq 1+x$ and $(x^q+y^q)^{1/q}\leq x+y$ for all $q>1$ and $x,y>0$, it yields the estimate for $\I^n$,
\begin{equation}\label{eq:In_stochastic}
	\begin{aligned}
		\I^n & \lesssim \left( 1 + \frac{(\|u\|_\infty^q+1)k^q+(\kappa k)^{q/2}}{\delta^q} \right)^{1/q} \|\theta^0\|_{L^q} \int_{t^n}^{t^{n+1}} \|\nabla\overline u(s)\|_{L^p}ds \\
		& \lesssim \left( 1+(\|u\|_\infty+1)\frac{k}{\delta} + \frac{\sqrt{\kappa k}}{\delta}\right)\|\theta^0\|_{L^q} \int_{t^n}^{t^{n+1}} \|\nabla\overline u(s)\|_{L^p}ds.
	\end{aligned}
\end{equation}
and hence by combining it with \eqref{eq:intermediate_bound_lemma_stochastic} and using that $u\in L^1(W^{1,p})$ we get the result stated by the Lemma.
\end{proof}

At this point we collect the three discretization errors (time, initial data, vector-field) from Lemmas~\ref{lemma:disc_time}--~\ref{lemma:disc_vectorfield}. 
Since the Kantorovich--Rubinstein distance $\D_\delta(\cdot,\cdot)$ satisfies the triangle inequality we can just write now for any $t\in[t^m,t^{m+1})$ and any $m\in\llb 0, N-1 \rrb$,
\[
\begin{aligned}
\D_\delta(\theta(t),\theta_{k,h}(t)) & \leq \D_\delta(\theta(t),\theta(t^m))+\D_\delta(\theta(t^m),\theta^h(t^m))\\
 & \qquad +\D_\delta(\theta^h(t^m),\theta^{k,h}(t^m))+\D_\delta(\theta^{k,h}(t^m),\theta_{k,h}(t^m)),
\end{aligned}
\]
where $\theta^{k,h}$ is the unique solution to the advection-diffusion equation \eqref{eq:advdiff} with vector field $u^k$ and initial datum $\theta_h^0$. Notice that Lemma \ref{lemma:disc_time} and \ref{lemma:disc_datum} yield control over the first two addends. In order to get the bound for the third addend we apply Lemma \ref{lemma:disc_vectorfield} with $\theta^0=\theta^0_h$, and hence we arrive to the expression
\begin{equation}\label{eq:triangle_D_step_1}
\D_\delta(\theta(t),\theta_{k,h}(t)) \lesssim 1 + \frac{h+k\|u\|_\infty+\sqrt{k\kappa}}{\delta} + \D_\delta(\theta^{k,h}(t^m),\theta_{k,h}(t^m)).
\end{equation} 
The last addend in \eqref{eq:triangle_D_step_1} corresponds to the so-called truncation error or error caused by the scheme. We will concentrate on it in the next section.

\subsection{Error due to the scheme}\label{ss:error_cheme}

Since we already studied the errors coming from the discretization of the initial datum and vector field, we can consider now the continuous problem \eqref{eq:advdiff} with vector field $u^k$ and initial datum $\theta_h^0$. Also we can assume that $t=t^m$ for some $m\in\llb 0,N \rrb$ such that we have $\theta(t,x) = \theta^{k,h}(t^m,x)$. However, for the sake of a clear notation, along this section we will write $\theta$ denoting $\theta^{k,h}$. 

We want to study the distance $\D_\delta(\theta(t^m),\theta_{k,h}(t^m))$ and in order to do so it is more convenient to consider a piecewise linear temporal approximation of $\theta_{k,h}$ defined by
\[
\hat{\theta}_{k,h}(t,x) = \frac{t-t^n}{k}\theta_K^{n+1}+\frac{t^{n+1}-t}{k}\theta_K^n \qquad \text{for a.e. } (t,x)\in [t^n,t^{n+1})\times K
\]
for all $K\in\cT$ and all $n\in\llb 0,N\rrb$. One can check that indeed for the time points of the mesh $t^n$ with $n\in\llb 0,N\rrb$ it holds $\theta_{k,h}(t^n) = \hat{\theta}_{k,h}(t^n)$ and hence no additional error term must be considered. This linear piecewise temporal approximation is particularly convenient because it is weakly differentiable and by construction it holds,
\[
\partial_t\hat{\theta}_{k,h}(t,x) = \frac{\theta_K^{n+1}-\theta_K^n}{k} \quad \text{for a.e. } (t,x)\in [t^n,t^{n+1})\times K.
\]
Therefore we can directly apply \eqref{eq:dtD} to obtain
\[
\frac{d}{dt}\D_\delta(\theta,\hat{\theta}_{k,h}) = \int_\Omega \nabla\zeta\cdot u\theta dx + \kappa\int_\Omega\zeta\Delta\theta dx - \frac{1}{k}\sum_K\int_K\zeta(\theta_K^{n+1}-\theta_K^n)dx.
\]
where $\zeta$ represents the optimal Kantorovich potential associated to the distance $\D_\delta(\theta,\hat{\theta}_{k,h})$.

For the last term in the right hand side we can use the definition of the upwind scheme \eqref{eq:scheme2} in an analogous process to what it is done with the continuous part. Then, after integration over $[t^n,t^{n+1})$ we get
\begin{equation}\label{eq:error_scheme}
\D_\delta(\theta(t^{n+1}),\hat{\theta}_{k,h}(t^{n+1}))-\D_\delta(\theta(t^n),\hat{\theta}_{k,h}(t^n)) = \I^n+\II^n+\III^n+\IV^n
\end{equation}
with
\begin{align}
\I^n &= \int_{t^n}^{t^{n+1}}\int_\Omega\nabla\zeta\cdot u(\theta-\theta_h^{n+1})dx\,dt, \\
\II^n &= \int_{t^n}^{t^{n+1}}\int_\Omega\nabla\zeta\cdot u\theta_h^{n+1}dx\,dt+k\sum_K\zeta_K^n\sum_{L\sim K}|K\edge L|u_{KL}^n\frac{\theta_K^{n+1}+\theta_L^{n+1}}{2}, \\
\III^n &= k\sum_K\zeta_K^n\sum_{L\sim K}|K\edge L||u_{KL}^n|\frac{\theta_K^{n+1}-\theta_L^{n+1}}{2}, \\
\IV^n &= \kappa\int_{t^n}^{t^{n+1}}\sum_K\int_K\zeta\left(\Delta\theta-\sum_{L\sim K} \frac{|K\edge L|}{|K|}\frac{\theta_L^{n+1}-\theta_K^{n+1}}{d_{KL}}\right)dx\,dt,
\end{align}
where we use the notation
\[
\zeta_K^n = \avint_{t^n}^{t^{n+1}}\avint_K \zeta dx\,dt.
\]

We will study the contribution to the final error caused by the scheme analysing the four terms separately in the four following Lemmas.

\begin{lemma}[Error from $\I^n$]\label{lemma:i}
The first contribution to the error caused by the scheme is
\[
\sum_n\I^n \lesssim 1+\frac{\sqrt{kT}\|u\|_\infty}{\delta}.
\]
\end{lemma}

We will omit the proof of this Lemma for the sake of brevity because the argument is completely analogous to the one in \cite[Lemma 7]{SchlichtingSeis18}: a combination of properties of the optimal transport distance, Morrey's inequality and stability estimates.

From now on our procedure here diverges from the techniques in \cite{SchlichtingSeis18}, providing indeed the better convergence rate for the size of the mesh $h$.
\begin{lemma}[Error from $\II^n$]\label{lemma:ii}
The second contribution to the error caused by the scheme is
\[
\sum_n\II^n \lesssim \frac{h}{\delta} \min\set*{ \norm{u}_\infty \sqrt{\frac{T}{\kappa}} , \sqrt{\frac{T\norm{u}_\infty}{h}}} .
\]
\end{lemma}

\begin{proof}
In order to proof the estimate for $\II^n$ first it is convenient to rewrite it in a more suitable way. Let us abuse the notation for the sake of a clear exposition of the results and write $u^n$ for $u(t^n)$ and $\zeta^n$ for the average of $\zeta$ over the interval $[t^n,t^{n+1})$. With this notation for the first addend in $\II^n$ notice that
\[
\int_K\nabla\zeta\cdot u^ndx = \sum_{L\sim K}\int_{K\edge L}\zeta u^n\cdot n_{KL}d\Ha^{d-1} - \int_K \zeta\nabla\cdot u^ndx.
\]
Meanwhile, since $u_{KL}^n = -u_{LK}^n$, for the second addend it holds
\[
k\sum_K\zeta_K^n\sum_{L\sim K}|K\edge L|u_{KL}^n\frac{\theta_K^{n+1}+\theta_L^{n+1}}{2} = k\sum_K\theta_K^{n+1}\sum_{L\sim K}|K\edge L|u_{KL}^n\frac{\zeta_K^n-\zeta_L^n}{2}.
\]
Therefore we can develop the whole term as follows
\[
\begin{aligned}
\II^n & = k\sum_K\theta_K^{n+1}\sum_{L\sim K}\left[ \int_{K\edge L} \zeta^nu^n\cdot n_{KL} d\Ha^{d-1}-|K\edge L| u_{KL}^n\frac{\zeta_K^n+\zeta_L^n}{2} \right] \\
& \qquad - k\sum_K\theta_K^{n+1}\int_K (\zeta^n-\zeta_K^n)\nabla\cdot u^n dx \\
& = k\sum_K\theta_K^{n+1}\sum_{L\sim K} \int_{K\edge L}\zeta^n \left[(u-u^n_K)-\avint_{K\edge L} (u^n-u_K^n)d\Ha^{d-1}\right]\cdot n_{KL} d\Ha^{d-1} \\
& \qquad + k\sum_K\theta_K^{n+1}\sum_{L\sim K} |K\edge L| u_{KL}^n\left( \avint_{K\edge L}\zeta^nd\Ha^{d-1}-\frac{\zeta_K^n+\zeta_L^n}{2} \right) \\
& \qquad - k\sum_K\theta_K^{n+1}\int_K (\zeta^n-\zeta_K^n)\nabla\cdot u^n dx, \\
\end{aligned}
\]
so that we call the three addends $\II^n_1$, $\II^n_2$ and $\II^n_3$ respectively. 

First, to estimate $\II^n_1$, we can use the fact that $\zeta_K^n$ is constant to add and subtract on each $L\sim K$ a term of the form
\[
\zeta_K^n\int_{K\edge L} (u^n-u_K^n)\cdot n_{KL}d\Ha^{d-1}
\]
such that we obtain
\[
\begin{aligned}
\II^n_1 & = k\sum_K\theta_K^{n+1}\sum_{L\sim K} \int_{K\edge L}(\zeta^n-\zeta_K^n) (u-u^n_K)\cdot n_{KL} d\Ha^{d-1} \\
 & \qquad + k\sum_K\theta_K^{n+1}\sum_{L\sim K} \int_{K\edge L}\zeta_K^n (u-u^n_K)\cdot n_{KL} d\Ha^{d-1} \\
 & \qquad - k\sum_K\theta_K^{n+1}\sum_{L\sim K} \int_{K\edge L}\zeta^n \avint_{K\edge L} (u^n-u_K^n)d\Ha^{d-1}\cdot n_{KL} d\Ha^{d-1} \\
& \leq 2k\sum_K\theta_K^{n+1}\|\zeta^n-\zeta_K^n\|_{L^\infty}\int_{\partial K}|u^n-u_K^n|d\Ha^{d-1}.
\end{aligned}
\]
Now on the one hand we use the Lipschitz condition of the Kantorovich potential, i.e. for every $x\in K$ it holds
\begin{equation}\label{eq:zeta:Lip}
|\zeta(x)-\zeta_K| = \left| \avint_K (\zeta(x)-\zeta(y))dy \right| \leq \|\nabla\zeta\|_{L^\infty}\avint_K |x-y|dy \lesssim \frac{h}{\delta}.
\end{equation}
On the other hand by means of the trace and the Poincaré inequality \eqref{eq:regularity_mesh} we obtain
\[
\int_{\partial K}|u^n-u_K^n|d\Ha^{d-1} \lesssim \|\nabla u^n\|_{L^1(K)} + \frac{1}{h}\|u^n-u^n_K\|_{L^1(K)} \lesssim \|\nabla u^n\|_{L^1(K)}.
\]
Therefore combining everything, summing over $n$ and using Hölder's inequality we get the estimate,
\[
\sum_n \II^n_1 \lesssim \frac{kh}{\delta}\sum_n\sum_K\theta_K^{n+1}\|\nabla u^n\|_{L^1(K)} \leq \frac{h}{\delta} \|\theta_{k,h}\|_{L^\infty(L^q)}\|\nabla u\|_{L^1(L^p)}.
\]

For $\II^n_2$ instead we use again that $u_{KL}^n = -u_{LK}^n$ and hence we can rewrite the term as
\[
\II_2^n = k\sum_K\frac{\theta_K^{n+1}-\theta_L^{n+1}}{2}\sum_{L\sim K} |K\edge L| u_{KL}^n\left( \avint_{K\edge L}\zeta^nd\Ha^{d-1}-\frac{\zeta_K^n+\zeta_L^n}{2} \right),
\]
so that by the Lipschitz property of $\zeta$ the last factor is bounded by
\[
\avint_{K\edge L}\zeta^nd\Ha^{d-1} - \frac{\zeta_K^n-\zeta_L^n}{2} \lesssim \avint_K\avint_{K\edge L} (\zeta^n(x)-\zeta(y))d\Ha^{d-1}(x)dy \lesssim \frac{h}{\delta}.
\]
Therefore, $\II^n_2$ is controlled then by a term of the form
\[
\II_2^n\lesssim \frac{kh}{\delta}\|u\|_\infty\sum_K\sum_{K\edge L} |K\edge L||\theta_K^{n+1}-\theta_L^{n+1}|,
\]
and thus a direct application of \eqref{eq:weakstrongBVspace} yields
\[
\sum_n\II^n_2 \lesssim \frac{h}{\delta}\min\set*{ \norm{u}_\infty \sqrt{\frac{T}{\kappa}} , \sqrt{\frac{T\norm{u}_\infty}{h}}}.
\]
Finally for the third addend $\II^n_3$ we make use again of the Lipschitz property of $\zeta$ from~\eqref{eq:zeta:Lip} and we bound the divergence of the vector field $u^n$ by its gradient and some dimension dependant constant such that we obtain
\[
\II^n_3 \lesssim \frac{k h}{\delta} \sum_{K} \abs[\big]{ \theta^{n+1}_K} \int_K \abs*{\nabla u^n} dx
\leq  \frac{kh}{\delta}\|\theta_h^{n+1}\|_{L^q}\|\nabla u^n\|_{L^p}.
\]
After summation in $n$ we get a bound analogous to the bound that we got for $\II^n_1$ and thus all three addends in $\II^n$ are controlled by the factor stated in the claim of the Lemma.
\end{proof}

\begin{lemma}[Error from $\III^n$]\label{lemma:iii}
The third contribution to the error caused by the scheme is
\[
\sum_n\III^n \lesssim \frac{h}{\delta} \min\set*{ \norm{u}_\infty \sqrt{\frac{T}{\kappa}} , \sqrt{\frac{T\norm{u}_\infty}{h}}} .
\]
\end{lemma}

\begin{proof}
The proof of this Lemma follows a similar strategy to what has been performed in the previous one. First of all notice that this time $|u_{KL}^n|=|-u_{KL}^n|$ and hence we can symmetrize $\III^n$ as
\[
\III^n = k\sum_K\sum_{L\sim K}|K\edge L||u_{KL}^n|\frac{\zeta_K^n-\zeta_L^n}{2}\frac{\theta_K^{n+1}-\theta_L^{n+1}}{2}.
\]
Since the Kantorovich potential is Lipschitz we have the bound
\[
|\zeta_K^n-\zeta_L^n| \leq \avint_K\avint_L |\zeta^n(x)-\zeta^n(y)|dxdy \leq \|\nabla\zeta\|_{L^\infty} \avint_K\avint_L |x-y|dxdy \lesssim \frac{h}{\delta}
\]
and hence
\[
\III^n \lesssim \frac{kh}{\delta}\|u\|_\infty\sum_K\sum_{L\sim K}|K\edge L||\zeta_K^n-\zeta_L^n|.
\]
After summation in $n$, by means of the BV estimate \eqref{eq:weakstrongBVspace} as before we obtain the statement of the Lemma.
\end{proof}

Finally, to study the contribution made by the diffusion term we will follow a similar technique to what it is done in Lemma \ref{lemma:disc_time} but adapting it now to the setting of a finite volume scheme. In order to make this suitable approximation of the Laplacian we need to argue as follows.

Given an admissible tessellation $\cT$ of $\Omega$ and two neighboring cells $K,L\in\cT$ we define a diffeomorphism $\phi_{KL}:K\rightarrow L$ with constant Jacobian derivative, what means 
\[
J\phi_{KL} \equiv |\det\nabla\phi_{KL}|  = \frac{|L|}{|K|}
\]
such that the mass is preserved. Since all the admissible cells are convex the existence of this map is guaranteed, for instance consider an appropriate Brenier map \cite{Brenier1991,McCann1995} or some other analogous construction \cite{AleskerDarMilman99}.  Then, using this diffeomorphism we can define a \emph{finite-volume-based} approximation of the Laplacian such as
\[
\Delta^hf(x) = \sum_{L\sim K}\frac{|K\edge L|}{|K|}\frac{f\circ\phi_{KL}(x)-f(x)}{d_{KL}} \quad \text{for a.e. } x\in K \text{ and all } K\in\cT.
\]
Indeed, for sufficiently regular functions $f$ it holds $\lim_{h\rightarrow 0}\Delta^hf = \Delta f$. This will be a key instrument in the proof of the next and last Lemma. 

\begin{lemma}[Error from $\IV^n$]\label{lemma:iv}
The fourth term does not contribute to the error caused by the scheme, that is
\[
\sum_n\IV^n \leq 0.
\]
\end{lemma}

\begin{proof}
To prove this result we follow an adapted version of the technique used in Lemma \ref{lemma:disc_time} that the authors explain in more detail in \cite{NavarroFernandezSchlichtingSeis21}. This technique in turn comes inspired by \cite{FournierPerthame19}. Let us start by considering an approximation of the Laplacian as explained in the previous paragraphs. By means of $\Delta^h$ we can also define an approximation to $\IV^n$ as follows,
\[
\begin{aligned}
\frac{1}{\kappa}\IV^n_h & = \int_{t^n}^{t^{n+1}}\sum_K\int_K\zeta\left(\Delta^h\theta-\sum_{L\sim K} \frac{|K\edge L|}{|K|}\frac{\theta_L^{n+1}-\theta_K^{n+1}}{d_{KL}}\right)dx\,dt\\
& = \int_{t^n}^{t^{n+1}}\sum_K\int_K\zeta(x)\left(\sum_{L\sim K} \frac{|K\edge L|}{|K|}\frac{\theta\circ\phi_{KL}(x)-\theta(x)}{d_{KL}}-\sum_{L\sim K} \frac{|K\edge L|}{|K|}\frac{\theta_L^{n+1}-\theta_K^{n+1}}{d_{KL}}\right)dx\,dt.
\end{aligned}
\]
Notice that since $\zeta\in W^{1,\infty}$ and $\theta\in W^{1,1}$ it holds
\[
\lim_{h\rightarrow 0} \IV^n_h = \IV^n,
\]
and thus it is enough to study the approximation $\IV^n_h$ instead of $\IV^n$.

Through a convenient change of variables $y=\phi_{KL}(x)$ on the first addend that we can make because $\phi_{KL}$ is a diffeomorphism, it yields
\[
\begin{aligned}
\frac{1}{\kappa}\IV^n_h & = \int_{t^n}^{t^{n+1}}\sum_K\sum_{L\sim K}\frac{|K\edge L|}{|K|}\frac{1}{d_{KL}}\int_K\zeta(x)\left[(\theta\circ\phi_{KL}(x)-\theta_L^{n+1})-(\theta(x)-\theta_K^{n+1})\right]dx\,dt \\
& = \int_{t^n}^{t^{n+1}}\sum_K\sum_{L\sim K}\frac{|K\edge L|}{|L|}\frac{1}{d_{KL}}\int_L\zeta\circ\phi_{KL}^{-1}(y)(\theta(y)-\theta_L^{n+1})dydt \\
& \qquad - \int_{t^n}^{t^{n+1}}\sum_K\sum_{L\sim K}\frac{|K\edge L|}{|K|}\frac{1}{d_{KL}}\int_K\zeta(x)(\theta(x)-\theta_K^{n+1})dx\,dt,
\end{aligned}
\]
where we have used that the Jacobian derivative of $\phi_{KL}$ is constant and equals $|L|/|K|$. Then notice that by definition $\zeta$ is the optimal Kantorovich potential for the distance between $\theta$ and $\theta_{k,h}$, i.e.
\[
\sum_K\int_K\zeta(x)(\theta(x)-\theta_K^{n+1})dx = \D_\delta(\theta,\theta_{k,h}).
\]
Furthermore, the optimal $\zeta$ is taken as the supremum over a set of functions where $\zeta\circ\phi_{KL}^{-1}$ also belongs to. Therefore it holds
\[
\sum_K\int_K\zeta\circ\phi_{KL}^{-1}(x)(\theta(x)-\theta_K^{n+1})dx \leq \D_\delta(\theta,\theta_{k,h})
\] 
and hence, after relabelling in a suitable way
\begin{equation}\label{eq:final_estimate_iv}
\frac{1}{\kappa}\IV^n_h \leq \int_{t^n}^{t^{n+1}}\sum_K\sum_{L\sim K}\frac{|K\edge L|}{|K|}\frac{1}{d_{KL}}(\D_\delta(\theta,\theta_{k,h})\vert_{K}-\D_\delta(\theta,\theta_{k,h})\vert_{K})dt = 0,
\end{equation}
where we denote by $\D_\delta(\theta,\theta_{k,h})\vert_{K}$ the restriction of the distance $\D_\delta(\theta,\theta_{k,h})$ to the subset $K\subset\Omega$. Since \eqref{eq:final_estimate_iv} holds uniformly in $h$, it yields that in the limit $\IV^n$ does not contribute to the error caused the scheme.
\end{proof}

\begin{proof}[Proof of Theorem \ref{theorem}]
Finally we can get the result on Theorem \ref{theorem} with a straightforward combination of Lemmas \ref{lemma:disc_time}$-$\ref{lemma:iv}. The first three of them already yield the intermediate estimate \eqref{eq:triangle_D_step_1}. For the remaining term we just notice that by definition $\theta^{k,h}(0)=\theta_{k,h}(0)=\theta_h^0$ and thus we can sum on \eqref{eq:error_scheme} so that
\[
\begin{aligned}
\D_\delta(\theta^{k,h}(t^m),\theta_{k,h}(t^m)) &= \sum_{n=0}^m (\I^n+\II^n+\III^n+\IV^n) \\
& \lesssim 1+\frac{h}{\delta}\min\set*{ \norm{u}_\infty \sqrt{\frac{T}{\kappa}} , \sqrt{\frac{T\norm{u}_\infty}{h}}}+\frac{\sqrt{kT}\|u\|_\infty}{\delta}.
\end{aligned}
\]
Combining this with \eqref{eq:triangle_D_step_1} we get the estimate on Theorem \ref{theorem}.
\end{proof}

\appendix
\section{Stochastic Lagrangian flows on bounded domains}\label{s:SLF}

The proof of Lemma~\ref{lemma:disc_datum} is based on a Lagrangian representation of the solution to~\eqref{eq:advdiff}. As discussed in the beginning of the proof of Lemma~\ref{lemma:disc_datum}, we can assume without loss of generality that the driving vector field $u$ is smooth, which we shall do throughout this appendix. Due to the presence of the Laplacian, the Lagrangian representation will be stochastic. Hence, we fix a filtered probability space $(U, \cF, \cF_t , \Prob)$ on which we define the according stochastic process. For a domain $\Omega$ having for each $x\in \partial \Omega$ a unique normal $n(x)$, it is well-know (see e.g.~\cite[Chapter 3.1]{Pilipenko2014}) that any smooth solution of~\eqref{eq:advdiff} is intimately related to the solution to the SDE
\begin{equation}\label{eq:SDE}
	d X_t = u(t,X_t) \, ds + \sqrt{2\kappa} \, d B_t - n(X_t) \, d L_t ,
\end{equation}
where $\set{B_t}_{t\geq 0}$ is an $\cF_t$-adapted Brownian motion on $\R^d$
and $\set{L_t}_{t\geq 0}$ is an $\cF_t$-adapted local time of the process $\set{X_t}_{t\geq 0}$ at the boundary $\partial \Omega$, that is a non-decreasing process with $L_0=0$ such that
\begin{equation}\label{eq:boundary_process_properties}
	\int_0^t \, dL_s \leq t , \qquad \int_0^t \dsOne_{X_s\not\in \partial \Omega} \, d L_s = 0 . 
\end{equation}
The representation is obtained via the Kolmogorov backward equation associated to~\eqref{eq:advdiff}, that is a solution $f:[0,t]\times \Omega\to \R$ of the backward parabolic equation with some terminal condition $g\in C^2(\Omega)$ satisfying
\begin{align}
	\partial_t   f + \kappa \Delta  f + u\cdot\nabla f &= 0 \qquad\text{ in } [0,t]\times \Omega , \nonumber \\
	\nabla f(s,x) \cdot n(x) &= 0  \qquad\text{ for } (s,x)\in [0,t]\times \partial \Omega ,\label{eq:backwardPDE} \\
	f(t,\cdot) &= g \qquad \text{ in } \Omega .  \nonumber
\end{align}
Then, any solution of~\eqref{eq:SDE} provides a solution to~\eqref{eq:backwardPDE} via the observable representation
\begin{equation}\label{eq:adj:stochRep}
	f(s,x) = \EX_{s,x}\pra*{g(X(t))} =  \EX\pra*{g(X(t)) \middle| X(s)=x} \qquad\text{for } (s,x) \in (0,t) \times \Omega .
\end{equation}
From here we arrive at measure-valued solutions to~\eqref{eq:advdiff} via duality, which we give in the following definition.
\begin{definition}[Measure-valued solution to~\eqref{eq:advdiff}]\label{def:weakFPE}
	A Borel curve $\theta = (\theta_t)_{t\in [0,T]} \subset \cM(\R^d)$ is a measure-valued solution to the advection-diffusion equation~\eqref{eq:advdiff} provided that
	\begin{equation}\label{eq:weakFPE:integrability}
		\int_0^T \io \bra*{ \kappa + \abs{u(t,\cdot)}} \, d|\theta_t(\cdot)| \, dt < \infty 
	\end{equation}
	and for all $f\in C^{1,2}([0,T]\times \overline \Omega)\cap \set{\partial_n f \equiv 0 \text{ on } \partial \Omega}$ and all $0\leq t_1 \leq t_2 \leq T$ it holds
	\begin{equation}\label{eq:FPE:weakform}
		\io f(t_2,\cdot) \, d\theta_{t_2} - \io f(t_1,\cdot) \, d\theta_{t_1} = \int_{t_1}^{t_2} \io \bra*{ \partial_t  + \kappa \Delta + u\cdot\nabla} f(t,x) \, d\theta_t(x) \, dt = 0 .
	\end{equation}
\end{definition}
By a standard density argument~\cite[Lemma 8.1.2]{AGS2008}, it holds that any measure-valued solution in the sense of Definition~\ref{def:weakFPE} admits a narrowly continuous representative, coinciding with $(\theta_t)_{t\in(0,T)}$ for a.e. $t\in (0,T)$, in the space $\cM(\R^d)$ conserving the mass, i.e. $\theta_t(\overline \Omega) = \theta_0(\overline \Omega)$ for all $t\in (0,T]$. Hence, we can without loss of generality consider narrowly continuous paths $(\theta_t)_{t\in (0,T)}\subset \cP(\overline \Omega)$ solution to the advection-diffusion equation in the sense of Definition~\ref{def:weakFPE}.

For smooth $u$, we find a unique classical solutions $f\in C^{1,2}([0,T]\times \Omega)$ to the system~\eqref{eq:backwardPDE} (see~\cite{Friedman1964}) with terminal value $g\in C^2(\Omega)$. In particular this identifies via~\eqref{eq:FPE:weakform}, becoming for $t_1=0$ and $t\in (0,T]$ the identity $\int g(\cdot) \, d\theta_t= \int f(0,\cdot) \, d\theta_0$ a unique family $(\theta_t)_{t\in [0,T]}\subseteq \cP(U)$.

Based on the stochastic representation~\eqref{eq:adj:stochRep}, we obtain the pathwise Lagrangian representation 
\begin{equation}\label{eq:SLF}
	\io g(\cdot) \,d\theta_t = \io \EX_{0,x}\pra*{ g(X_t)} \, d\theta_0 = \EX\pra*{ g(X_t) \,\middle|\, \operatorname{law} X_0 = \theta_0  } = \EX_{\theta_0}\pra*{ g(X_t) } .
\end{equation}
\begin{remark}[Stochastic Lagrangian flows]
	For the case of $\Omega=\R^d$, i.e. no reflection, the Lagrangian representation~\eqref{eq:SLF} was obtained in~\cite{Figalli08}, for bounded coefficients and in~\cite{Trevisan2016} under the sole integrability condition~\eqref{eq:weakFPE:integrability}. The identification, also called stochastic Lagrangian flow, is based on martingale solutions to~\eqref{eq:SDE}, which is a weak solution concept for SDEs going back to~\cite{StroockVaradhan1969I,StroockVaradhan1969II}. 
	
	It seems also possible to directly generalize the concept of stochastic Lagrangian flows on a bounded set with reflecting boundary conditions to measurable vectorfields just satisfying~\eqref{eq:weakFPE:integrability}. Here, one would use the martingale problem formulation from~\cite{StroockVaradhan1971} for the reflected SDE~\eqref{eq:SDE} (see also~\cite[Chapter 3.2]{Pilipenko2014}) and do similar approximation steps as outlined in~\cite[Appendix A]{Trevisan2016}. 
\end{remark}

\section*{Acknowledgement} 
The authors thank Christian Seis for many helpful comments. This work is funded by the Deutsche Forschungsgemeinschaft (DFG, German Research Foundation) under Germany's Excellence Strategy EXC 2044 --390685587, Mathematics M\"unster: Dynamics--Geometry--Structure and by the DFG Grant 432402380.

\bibliographystyle{acm}
\bibliography{diffusive-mixing}
\end{document}